\documentclass[11pt]{article}

\usepackage{amsmath, amsfonts, amssymb, amsthm,  graphicx, enumerate}
\usepackage[letterpaper, left=1truein, right=1truein, top = 1truein, bottom = 1truein]{geometry}

\usepackage[numbers, square]{natbib}

\usepackage[
            CJKbookmarks=true,
            bookmarksnumbered=true,
            bookmarksopen=true,
            colorlinks=true,
            citecolor=red,
            linkcolor=blue,
            anchorcolor=red,
            urlcolor=blue,
            ]{hyperref}

\usepackage[usenames,dvipsnames]{color}
\usepackage{xcolor}

\usepackage[ruled, vlined, lined, commentsnumbered]{algorithm2e}

\usepackage{prettyref,soul,xspace}

\usepackage{tikz}
\usepackage{pgfplots,makecell}

\newcommand{\bem}{\begin{bmatrix}}
\newcommand{\eem}{\end{bmatrix}}

\newcommand{\reals}{{\mathbb{R}}}

\newcommand{\naturals}{{\mathbb{N}}}

\newcommand{\Expect}{\mathbb{E}}

\newcommand{\Prob}{\mathbb{P}}

\newcommand{\argmin}{\mathop{\rm argmin}}
\newcommand{\argmax}{\mathop{\rm argmax}}

\newcommand{\Th}{{^{\rm th}}}

\theoremstyle{remark}
\newtheorem{remark}{Remark}
\theoremstyle{plain}

\newtheorem{lemma}{Lemma}
\newtheorem{theorem}{Theorem}

\theoremstyle{definition}

\newtheorem{example}{Example}

\theoremstyle{plain}

\newrefformat{eq}{(\ref{#1})}
\newrefformat{chap}{Chapter~\ref{#1}}
\newrefformat{sec}{Section~\ref{#1}}
\newrefformat{algo}{Algorithm~\ref{#1}}
\newrefformat{fig}{Fig.~\ref{#1}}
\newrefformat{tab}{Table~\ref{#1}}
\newrefformat{rmk}{Remark~\ref{#1}}
\newrefformat{clm}{Claim~\ref{#1}}
\newrefformat{def}{Definition~\ref{#1}}
\newrefformat{cor}{Corollary~\ref{#1}}
\newrefformat{lmm}{Lemma~\ref{#1}}
\newrefformat{lemma}{Lemma~\ref{#1}}
\newrefformat{prop}{Proposition~\ref{#1}}
\newrefformat{app}{Appendix~\ref{#1}}
\newrefformat{ex}{Example~\ref{#1}}
\newrefformat{exer}{Exercise~\ref{#1}}
\newrefformat{soln}{Solution~\ref{#1}}
\newrefformat{cond}{Condition~\ref{#1}}

\newcommand{\lunder}[1]{{\underset{\raise0.3em\hbox{$\smash{\scriptscriptstyle-}$}}{#1}}}

\newcommand{\floor}[1]{{\left\lfloor {#1} \right \rfloor}}
\newcommand{\ceil}[1]{{\left\lceil {#1} \right \rceil}}

\newcommand{\norm}[1]{\|{#1} \|}

\newcommand{\fnorm}[1]{\|#1\|_{\rm F}}

\newcommand{\opnorm}[1]{\|#1\|_{\rm op}}

\newcommand{\iprod}[2]{\left \langle #1, #2 \right\rangle}

\newcommand{\indc}[1]{{\mathbf{1}_{\left\{{#1}\right\}}}}

\def\innergetnumber#1[#2]#3{#2}
\def\getnumber{\expandafter\innergetnumber\jobname}

\newcommand{\bbQ}{{\mathbb{Q}}}

\newcommand{\calP}{{\mathcal{P}}}

\newcommand{\tz}{{\tilde{z}}}

\newcommand{\renyi}{R\'enyi\xspace}

\newcommand{\pth}[1]{\left( #1 \right)}
\newcommand{\qth}[1]{\left[ #1 \right]}
\newcommand{\sth}[1]{\left\{ #1 \right\}}

\newtheorem{corollary}{Corollary}

\newtheorem*{thma}{Condition N}

\newcommand{\E} {\mathbb{E}}
\newcommand{\p} {\mathbb{P}}

\DeclareMathOperator*{\ave}{\operatorname{ave}}

\newcommand{\dcn}{degree correction}

\newcommand{\folds}{M}
\newcommand{\nmin}{n_{\min}}
\newcommand{\nmax}{n_{\max}}
\newcommand{\ztrue}{z}

\title{
Community Detection in Degree-Corrected Block Models
}

\author{Chao Gao$^1$, Zongming Ma$^2$, Anderson Y.~Zhang$^1$ and Harrison H.~Zhou$^1$\\
~\\
$^1$Yale University and $^2$University of Pennsylvania
}
\date{~}

\begin{document}
\maketitle

\begin{abstract}
Community detection is a central problem of network data analysis. 
Given a network, the goal of community detection is to partition the network nodes into a small number of clusters, which could often help reveal interesting structures.
The present paper studies community detection in Degree-Corrected Block Models (DCBMs). 
We first derive asymptotic minimax risks of the problem for a misclassification proportion loss under appropriate conditions.
The minimax risks are shown to depend on degree-correction parameters, community sizes, and average within and between community connectivities in an intuitive and interpretable way. 
In addition,
we propose a polynomial time algorithm to adaptively perform consistent and even asymptotically optimal community detection in DCBMs.

\smallskip

\end{abstract}
 



\section{Introduction}
\label{sec:intro}

In many fields such as social science, neuroscience and computer science, it has become increasingly important to process and make inference on relational data.
The analysis of network data, a prevalent form of relational data, becomes an important topic for statistics and machine learning.
One central problem of network data analysis is \emph{community detection}: to partition the nodes in a network into subsets. 
A meaningful partition of nodes can often uncover interesting information that is not apparent in a complicated network. 

An important line of research on community detection is based on Stochastic Block Models (SBMs) \citep{holland83}.
For any $p\in [0,1]$, let $\text{Bern}(p)$ be the Bernoulli distribution with success probability $p$.
Under an SBM with $n$ nodes and $k$ communities, given a symmetric probability matrix $B = (B_{uv}) =B^T\in[0,1]^{k\times k}$ and a label vector $z = (z(1),\dots,z(n))^T\in[k]^n$, where $[k] = \{1,\dots, k\}$ for any $k\in \naturals$, 
its adjacency matrix $A = (A_{ij})\in \{0,1\}^{n\times n}$, with ones encoding edges, is assumed to be symmetric with zero diagonals and $A_{ij}=A_{ji}\stackrel{ind.}{\sim} \text{Bern}(B_{z(i)z(j)})$ for all $i>j$.
In other words, the probability of an edge connecting any pair of nodes only depends on their community memberships.
To date, researchers in physics, computer science, probability theory and statistics have gained great understanding on community detection in SBMs.
See, for instance, \cite{bickel09,decelle2011asymptotic,mossel2012stochastic,mossel2013proof,massoulie2014community,mossel2014consistency,abbe2014exact,hajek2014achieving,chin2015stochastic,hajek2015achieving,abbe2015community,gao2015achieving,zhang2015minimax} and the references therein.
Despite a rich literature dedicated to their theoretical properties, SBMs suffer significant drawbacks when it comes to modeling real world social and biological networks.
In particular, due to the model assumption, all nodes within the same community in an SBM are exchangeable and hence have the same degree distribution.
In comparison, nodes in real world networks often exhibit degree heterogeneity even when they belong to the same community \cite{peixoto2015model}.
For example, Bickel and Chen \citep{bickel09} showed that for a karate club network, SBM does not provide a good fit for the data set, and the resulting clustering analysis is qualitatively different from the truth.

One way to accommodate degree heterogeneity is to introduce a set of degree-correction parameters $\sth{\theta_i:i=1,\dots,n}$, one for each node, which can be interpreted as the popularity or importance of a node in the network. 
Then one could revise the edge distributions to $A_{ij} = A_{ji}\stackrel{ind.}{\sim} \text{Bern}(\theta_i\theta_jB_{z(i)z(j)})$ for all $i>j$, and this gives rise to the Degree-Corrected Block Models (DCBMs) \cite{dasgupta2004spectral,karrer2011stochastic}.
In a DCBM, within the same community, a node with a larger value of degree-correction parameter is expected to have more connections than that with a smaller value.
On the other hand, SBMs are special cases of DCBMs in which the degree-correction parameters are all equal.
Empirically, the larger class of DCBMs is able to provide possibly much better fits to many real world network datasets \cite{peixoto2015model}. 
Since the proposal of the model, there have been various methods proposed for community detection in DCBMs, including but not limited to spectral clustering \citep{qin2013regularized, lei2015consistency, jin2015fast, gulikers2015spectral} and modularity based approaches \cite{karrer2011stochastic, zhao2012consistency, amini2013pseudo,chen2015convexified}.
On the theoretical side, \cite{gulikers2015impossibility} provides an information-theoretic characterization of the impossibility region of community detection for DCBMs with two clusters, and sufficient conditions have been given in \citep{zhao2012consistency,chen2015convexified} for strongly and weakly consistent community detection.
However, two fundamental statistical questions remain unanswered:
\begin{itemize}
\item What are the fundamental limits of community detection in DCBMs? 
\item Once we know these limits, can we achieve them adaptively via some polynomial time algorithm?
\end{itemize}
The answer to the first question can provide important benchmarks for comparing existing approaches and for developing new procedures. 
The answer to the second question can lead to new practical methodologies with theoretically justified optimality.
The present paper is dedicated to provide answers to these two questions.


\paragraph{Main contributions}
Our main contributions are two-folded. 
First, we carefully formulate community detection in DCBMs as a decision-theoretic problem and then work out its asymptotic minimax risks with sharp constant in the exponent under certain regularity conditions.
For example, let $k$ be a fixed constant. 
Suppose there are $k$ communities all about the same size $n/k$ and the average within community and between community edge probabilities are $p$ and $q$ respectively with $p>q$ and $p/q = O(1)$, 
then under mild regularity conditions, the minimax risk under the loss function that counts the proportion of misclassified nodes takes the form 
\begin{equation}
\qth{\frac{1}{n}\sum_{i=1}^n\exp\left(-\theta_i\frac{n}{k}(\sqrt{p}-\sqrt{q})^2\right)}^{1+ o(1)}
\label{eq:minimax-form}
\end{equation}
as $n\to\infty$ whenever it converges to zero and the maximum expected node degree scales at a sublinear rate with $n$.
The general fundamental limits to be presented in \prettyref{sec:main_result} allow the community sizes to differ and the number of communities $k$ to grow to infinity with $n$.
To the best our knowledge, this is the first minimax risk result for community detection in DCBMs. 
The minimax risk (\ref{eq:minimax-form}) has an intuitive form.
In particular, the $i\Th$ term in the summation can be understood as the probability of the $i\Th$ node being misclassified. 
When $\theta_i$ is larger, the chance of the node being misclassified gets smaller as it has more edges and hence more information of its community membership is available in the network.
The term $n/k$ is roughly the community size. Since the community detection problem can be reduced to a hypothesis testing problem with $n/k$ as its effective sample size, a larger $k$ implies a more difficult problem.
Furthermore, $(\sqrt{p}-\sqrt{q})^2$ reflects the degree of separation among the $k$ clusters.
Note that $p$ and $q$ are the average within and between community edge probabilities and so $(\sqrt{p}-\sqrt{q})^2$ measures the difference of edge densities within and between communities.
If the clusters are more separated in the sense that the within and between community edge densities differ more, the chance of each node being misclassified becomes smaller. 
When the degree-correction parameters are all equal to one and $p=o(1)$, the expression in (\ref{eq:minimax-form}) reduces to the minimax risk of community detection in SBMs in \cite{zhang2015minimax}.

In addition, we investigate computationally feasible algorithms for adaptively achieving minimax optimal performance.
In particular, we propose a polynomial time two-stage algorithm. 
In the first stage, we obtain a relatively crude community assignment via a weighted $k$-medians procedure on a low-rank approximation to the adjacency matrix. 
Working with a low-rank approximation (as opposed to 
the leading eigenvectors of the adjacency matrix) enables us to avoid common eigen-gap conditions needed to establish weak consistency for spectral clustering methods.
Based on result of the first stage, the second stage applies a local optimization to improve on the community assignment of each network node. 
Theoretically, we show that it can adaptively achieve asymptotic minimax optimal performance for a large collection of parameter spaces.
The empirical effectiveness of the algorithm is illustrated 
by simulation.


\paragraph{Connection to previous work}
The present paper is connected to a number of papers on community detection in DCBMs and SBMs.

It is connected to the authors' previous work on minimax community detection in SBMs \cite{gao2015achieving,zhang2015minimax}.
However, the involvement of degree-correction parameters poses significant new challenges.
For the study of fundamental limits, especially minimax lower bounds, the fundamental two-point testing problem in DCBMs compares two product probability distributions with 
different marginals, while in SBMs, the two product distributions can be divided to two equal sized blocks within which the marginals are the same.
Consequently, a much more refined Cram\'{e}r--Chernoff argument is needed to establish the desired bound.
{In addition, to establish matching minimax upper bounds, the analysis of the maximum likelihood estimators is technically more challenging than that in \cite{zhang2015minimax}
due to the presence of degree-correction parameters and the wide range in which they can take values. 
In particular, we use a new folding argument to obtain the desired bounds.}
For adaptive estimation, the degree-correction parameters further increase the number of nuisance parameters.
As a result, although we still adopt a ``global-to-local'' two-stage strategy to construct the algorithm, 
neither stage of the proposed algorithm in the present paper can be borrowed from the algorithm proposed in \cite{gao2015achieving}. 
We will give more detailed comments on the first stage below.
For the second stage, the penalized neighbor voting approach in \cite{gao2015achieving} requires estimation of degree-correction parameters with high accuracy and hence is infeasible. 
We propose a new \emph{normalized} neighbor voting procedure to avoid estimating $\theta_i$'s.

The first stage of the proposed algorithm is connected to the literature on spectral clustering, especially \cite{jin2015fast}.
The novelty in our proposal is that we cluster the rows of a low-rank approximation to the adjacency matrix directly as opposed to the rows of the matrix containing the leading eigenvectors of the adjacency matrix.
As a result, the new spectral clustering algorithm does not require any eigen-gap condition to achieve consistency.

\paragraph{Organization}
After a brief introduction to common notation, the rest of the paper is organized as follows. 
\prettyref{sec:main_result} presents the decision-theoretic formulation of community detection in DCBMs and derives matching asymptotic minimax lower and upper bounds under appropriate conditions.
Given the fundamental limits obtained, we propose in \prettyref{sec:algo} a polynomial time two-stage algorithm and study when a version of it can adaptively achieve minimax optimal rates of convergence.
The finite sample performance of the proposed algorithm is examined in \prettyref{sec:nume} on simulated data examples.
Some proofs of the main results are given in \prettyref{sec:proof} with additional proofs deferred to the appendices.

\paragraph{Notation}
For an integer $d$, we use $[d]$ to denote the set $\{1,2,...,d\}$. For a positive real number $x$, $\ceil{x}$ is the smallest integer no smaller than $x$ and $\floor{x}$ is the largest integer no larger than $x$. For a set $S$, we use $\indc{S}$ to denote its indicator function and $|S|$ to denote its cardinality. For a vector $v\in\mathbb{R}^d$, its norms are defined by $\norm{v}_1=\sum_{i=1}^n|v_i|$, $\norm{v}^2=\sum_{i=1}^nv_i^2$ and $\norm{v}_{\infty}=\max_{1\leq i\leq n}|v_i|$. For two matrices $A,B\in\mathbb{R}^{d_1\times d_2}$, their trace inner product is defined as $\iprod{A}{B}=\sum_{i=1}^{d_1}\sum_{j=1}^{d_2}A_{ij}B_{ij}$. The Frobenius norm and the operator norm of $A$ are defined by $\fnorm{A}=\sqrt{\iprod{A}{A}}$ and $\opnorm{A}=s_{\max}(A)$, where $s_{\max}(\cdot)$ denotes the largest singular value.



\section{Fundamental Limits}
\label{sec:main_result}

In this section, we present fundamental limits of community detection in 
DCBMs.
We shall first define an appropriate parameter space and a loss function. 
A characterization of asymptotic minimax risks then follows.

\subsection{Parameter Space and Loss Function}

Recall that a random graph of size $n$ generated by a DCBM has its adjacency matrix $A$ satisfying $A_{ii} = 0$ for all $i\in [n]$ and
\begin{equation}
	\label{eq:DCBM}
A_{ij} = A_{ji} \stackrel{ind}{\sim}
\text{Bern}(\theta_i\theta_jB_{z(i)z(j)})\,\,\,\,\mbox{for all $i\neq j\in [n]$}.
\end{equation}
For each $u\in [k]$ and a given $z\in [k]^n$, we let $n_u = n_u(z) = \sum_{i=1}^n \indc{z(i)=u}$ be the size of the $u\Th$ community.
Let $P = \Expect[A]\in [0,1]^{n\times n}$.
We propose to consider the following parameter space for DCBMs of size $n$:
\begin{equation}
\label{eq:DCBM-space}
\begin{aligned}
\calP_n(\theta,p,q,k,\beta;\delta) 
= \big\{& P\in [0,1]^{n\times n}: \exists z\in [k]^n\,\,\mbox{and}\,\, B = B^T\in \reals^{k\times k},\\
&  \mbox{s.t.~} P_{ii} = 0, P_{ij} = \theta_i\theta_jB_{z(i)z(j)},\, \forall i\neq j\in [n],\\
& \frac{1}{n_u}\sum_{z(i)=u} \theta_i \in [1-\delta, 1+\delta], \, \forall u\in [k], \\
& \max_{u\neq v}B_{uv}\leq q< p\leq \min_{u}B_{uu},\\
& \frac{n}{\beta k}-1 \leq n_u \leq \frac{\beta n}{k}+1,\,\, \forall u \in [k]\big\}. 
\end{aligned}
\end{equation}

We are mostly interested in the behavior of minimax risks over a sequence of such parameter spaces as $n$ tends to infinity and the key model parameters $\theta,p,q,k$ scale with $n$ in some appropriate way. 
On the other hand, we take $\beta \geq 1$ as an absolute constant and require the (slack) parameter $\delta$ to be an $o(1)$ sequence throughout the paper.

To see the rationale behind the definition in \eqref{eq:DCBM-space}, let us examine each of the parameters used in the definition.
The starting point is $\theta\in \reals_+^n$, which we treat for now as a given sequence of {\dcn} parameters.
Given $\theta$, we consider all possible label vectors $z$ such that the approximate normalization $\frac{1}{n_u}\sum_{z(i)=u}\theta_u =1 + o(1)$ holds for all communities.
{The introduction of the slack parameter $0<\delta=o(1)$ rules out those parameter spaces in which community detection can be trivially achieved by only examining the normalization of the $\theta_i$'s.}
On the other hand, the proposed normalization ensures that for all $u\neq v\in [k]$,
\begin{equation*}
B_{uu} \approx \frac{1}{n_u(n_u-1)}\sum_{i:z(i)=u}\sum_{j\neq i: z(j)=u}
P_{ij}
\quad \mbox{and} \quad
B_{uv} \approx \frac{1}{n_u n_v}\sum_{i:z(i)=u}\sum_{j: z(j)=v} P_{ij}.
\end{equation*}
Therefore, $B_{uu}$ and $B_{uv}$
can be understood as the (approximate) average connectivity within the $u\Th$ community and between the $u\Th$ and the $v\Th$ communities, respectively.
Under this interpretation, $p$ can be seen as a lower bound on the within community connectivities and $q$ an upper bound on the between community connectivities.
We require the assumption $p>q$ to ensure that the model is ``assortative'' in an average sense.
Finally, we also require the individual community sizes to be contained in the interval $[n/(\beta k)-1,\beta n/k+1]$. In other words, the community sizes are assumed to be of the same order.

\begin{remark}
	\vskip -0.75em
An interesting special case of the parameter space in \eqref{eq:DCBM-space} is when $\theta = 1_n$, where $1_n\in\mathbb{R}^n$ is the all one vector. In this case, the parameter space reduces to one for assortative stochastic block models.	
\end{remark}

\vskip -0.75em
As for the loss function, we use the following misclassification proportion that has been previously used in the investigation of community detection in stochastic block models \cite{zhang2015minimax,gao2015achieving}:
\begin{equation}
	\label{eq:loss}
\ell(\hat{z}, z)=\frac{1}{n}\min_{\pi\in\Pi_k} H(\hat{z}, \pi(z)),
\end{equation}
where $H(\cdot,\cdot)$ is the Hamming distance defined as $H(z_1,z_2)=\sum_{i\in[n]}\indc{z_1(i)\neq z_2(i)}$ 
and $\Pi_k$ is the set of all permutations on $[k]$.
Here, the minimization over all permutations is necessary since we are only interested in comparing the partitions resulting from $z$ and $\hat{z}$ and so the actual labels used in defining the partitions should be inconsequential.

\subsection{Minimax Risks}

Now we study the minimax risk of the problem
\begin{equation}
\inf_{\hat{z}}\sup_{\calP_n(\theta,p,q,k,\beta;\delta)} \Expect\, \ell(\hat{z},z).\label{eq:def-mmm}
\end{equation}
In particular, we characterize the asymptotic behavior of (\ref{eq:def-mmm}) as a function of $n,\theta,p,q,k$ and $\beta$. 
The key information-theoretic quantity that governs the minimax risk of community detection is $I$, which is defined through 
\begin{equation}
\exp(-I)=\begin{cases}\frac{1}{n}\sum_{i=1}^n\exp\left(-\theta_i\frac{n}{2}(\sqrt{p}-\sqrt{q})^2\right), & k=2, \\
\frac{1}{n}\sum_{i=1}^n\exp\left(-\theta_i\frac{n}{\beta k}(\sqrt{p}-\sqrt{q})^2\right),& k\geq 3.
\end{cases}
\label{eq:formula-I}
\end{equation}
Note that $I$ depends on $n$ not only directly but also through $\theta$, $p$, $q$ and $k$.

\paragraph{Minimax upper bounds}
Given any parameter space $\calP_n(\theta,p,q,k,\beta;\delta)$,
we can define the following estimator:
\begin{equation}
\label{eq:mle}
\begin{aligned}
\hat{z} = \argmax_{z'\in \calP_n(\theta,p,q,k,\beta;\delta)} \prod_{1\leq i<j \leq n} 
\big[&(\theta_i\theta_j p)^{A_{ij}}(1-\theta_i\theta_j p)^{1-A_{ij}} \indc{z'(i)=z'(j)}  \\
&+ (\theta_i\theta_j q)^{A_{ij}}(1-\theta_i\theta_j q)^{1-A_{ij}} \indc{z'(i)\neq z'(j)} 
\big].
\end{aligned}
\end{equation}
If there is a tie, we break it arbitrarily.
The estimator \eqref{eq:mle} is the maximum likelihood estimator for a special case of DCBM where $B_{uu}= p$ and $B_{uv}=q$ for all $u\neq v\in [k]$.
In other cases, the objective function in \eqref{eq:mle} is a misspecified likelihood function.
For any sequences $\{a_n\}$ and $\{b_n\}$, we write $a_n = \Omega(b_n)$ if $a_n\geq C b_n$ for some absolute constant $C>0$ for all $n\geq 1$.
The following theorem characterizes the asymptotic behavior of the risk bounds for the estimator \eqref{eq:mle}.

\begin{theorem}[Minimax Upper Bounds]
	\label{thm:upper-nbar}
Consider any sequence \\
$\{\calP_n(\theta,p,q,k,\beta;\delta)\}_{n=1}^\infty$ such that as $n\to\infty$, $I\to\infty$, $p>q$, $\|\theta\|_{\infty} = o(n/k)$,
$\min_{i\in[n]}\theta_i=\Omega(1)$ and $\log k = o(\min(I,\log{n}))$. 
When $k\geq 3$, further assume $\beta\in[1,\sqrt{5/3})$.
Then the estimator in \eqref{eq:mle} satisfies
\begin{equation*}
\limsup_{n\to\infty}\, \frac{1}{I}\log\bigg(\sup_{\calP_n(\theta,p,q,k,\beta;\delta)} \Expect\, \ell(\hat{z},z) \bigg) \leq -1.
\end{equation*}
\end{theorem}

Before proceeding, we briefly discuss the conditions in
\prettyref{thm:upper-nbar}.
First, the condition $\min_{i\in[n]}\theta_i=\Omega(1)$ requires that all $\theta_i$'s are at least of constant order. 
One should note that this condition does not rule out the possibility that $\max_i\theta_i \gg \min_i\theta_i$, and so a great extent of degree variation, even within the same community, is allowed.
Next, $\log k = o(\log n)$ requires that the number of communities $k$, if it diverges to infinity, grows at a sub-polynomial rate with the number of nodes $n$.
Furthermore, $\beta\in[1,\sqrt{5/3})$ is a technical condition that we need for a combinatorial argument in the proof to go through when $k\geq 3$. 
{Informed readers might find the result in \prettyref{thm:upper-nbar} in parallel to that in \cite{zhang2015minimax}. 
However, due to the presence of degree-correction parameters, the proof of \prettyref{thm:upper-nbar} is significantly different from that of the corresponding result in \cite{zhang2015minimax}.
For example, a new folding argument is employed to deal with degree heterogeneity.}

\paragraph{Minimax lower bounds}
We now show that the rates in \prettyref{thm:upper-nbar} are asymptotic minimax optimal by establishing matching minimax lower bounds.
To this end, we require the following condition on 
the degree-correction parameters $\theta\in\mathbb{R}_+^n$. 
The condition guarantees that $\calP_n(\theta,p,q,k,\beta;\delta)$ is non-empty.
Moreover, it is only needed for establishing minimax lower bounds.

\begin{thma}
We say that $\theta\in\mathbb{R}_+^n$ satisfies Condition N if
\begin{enumerate}
\item{When $k=2$,}
there exists a disjoint partition $\mathcal{C}_1,\mathcal{C}_2$ of $[n]$, such that $|\mathcal{C}_1|=\floor{n/2}$,  $|\mathcal{C}_2|\in\{\floor{n/2},\floor{n/2}+1\}$ and $|\mathcal{C}_u|^{-1}\sum_{i\in\mathcal{C}_u}\theta_i\in (1-\delta/4,1+\delta/4)$ for $u=1,2$.
\item{When $k\geq 3$,} there exists a disjoint partition $\{\mathcal{C}_u\}_{u\in[k]}$ of $[n]$, such that $|\mathcal{C}_1|\leq |\mathcal{C}_2|\leq...\leq |\mathcal{C}_k|$, $|\mathcal{C}_1|=|\mathcal{C}_2|=\floor{n/(\beta k)}$ and $|\mathcal{C}_u|^{-1}\sum_{i\in\mathcal{C}_u}\theta_i\in (1-\delta/4,1+\delta/4)$ for all $u\in[k]$.
\end{enumerate}
\end{thma}
We note that the condition is only on $\theta$ (as opposed to the parameter space) and the actually communities in the data generating model
need not coincide with the partition that occurs in the statement of the condition.

With the foregoing definition, we have the following result.
\begin{theorem}[Minimax Lower Bounds]
	\label{thm:lower-nbar}
Consider any sequence \\
$\{\calP_n(\theta,p,q,k,\beta;\delta)\}_{n=1}^\infty$ such that as $n\to\infty$, $I\to\infty$, $1< p/q=O(1)$, $p\|\theta\|_\infty^2 = o(1)$, $\log k = o(I)$, $\log(1/\delta) = o(I)$ and $\theta$ satisfies Condition N.
Then 
\begin{equation*}
\liminf_{n\to\infty}\, \frac{1}{I}\log\bigg(\inf_{\hat{z}}\sup_{\calP_n(\theta,p,q,k,\beta;\delta)} \Expect\, \ell(\hat{z},z) \bigg) \geq -1.
\end{equation*}
\end{theorem}

Compared with the conditions in \prettyref{thm:upper-nbar}, the conditions of \prettyref{thm:lower-nbar} are slightly different.
The condition $1< p/q=O(1)$ ensures that the smallest average within community connectivity is of the same order as (albeit larger than) the largest average between community connectivity. 
Such an assumption is typical in the statistical literature on block models.
The condition $\norm{\theta}^2_{\infty}p=o(1)$ ensures that the maximum expected node degree scales at a sublinear rate with the network size $n$.
Furthermore, 
when $k=O(1)$, the condition $\log k=o(I)$ can be dropped because it is equivalent to $I\rightarrow\infty$, 
which in turn 
is necessary for the minimax risk to converge to zero.

Combining both theorems, we have the minimax risk of the problem.
\begin{corollary}
	\label{cor:minimax}
Under the conditions of Theorems \ref{thm:upper-nbar} and \ref{thm:lower-nbar}, we have
$$
\inf_{\hat{z}}\sup_{{\mathcal{P}}_n(\theta,p,q,k,\beta; \delta)}\mathbb{E} \ell(\hat{z},z)=\exp(-(1+o(1))I),
$$
where $o(1)$ stands for a sequence whose absolute values tend to zero as 
$n$ tends to infinity.
\end{corollary}

Setting $\beta=1$ in \prettyref{cor:minimax} leads to the minimax result \eqref{eq:minimax-form} in the introduction.
We refer to \prettyref{sec:intro} for the meanings of the terms in $I$. 


\begin{remark}
	\vskip -0.75em
When $\theta = 1_n$, the foregoing minimax risk reduces to the corresponding result for stochastic block models \cite{zhang2015minimax} in the sparse regime where $q<p=o(1)$.
In this case, \eqref{eq:formula-I} implies that the minimax risk is
$$
\exp(-(1+o(1))I) = \begin{cases}\exp\left(-(1+o(1))\frac{n}{2} (\sqrt{p}-\sqrt{q})^2\right), & k=2,\\
\exp\left(-(1+o(1))\frac{n}{\beta k}(\sqrt{p}-\sqrt{q})^2\right), & k\geq 3.
\end{cases}
$$
Note that when $q<p=o(1)$, the \renyi divergence of order $\frac{1}{2}$ used in the minimax risk expression in \cite{zhang2015minimax} is equal to $(1+o(1))(\sqrt{p}-\sqrt{q})^2$.
\end{remark}

\section{An Adaptive and Computationally Feasible Procedure}
\label{sec:algo}

Theorem \ref{thm:upper-nbar} shows that the minimax rate can be achieved by the estimator \eqref{eq:mle} obtained via combinatorial optimization which is not computationally feasible.
Moreover, the procedure depends on the knowledge of the parameters $\theta$, $p$ and $q$.
These features make it not applicable in practical situations.
In this section, we introduce a two-stage algorithm for community detection in DCBMs which is not only computationally feasible but also adaptive over a wide range of unknown parameter values.
We show that the procedure achieves
minimax optimal rates under certain regularity conditions.

\subsection{A Two-Stage Algorithm}

The proposed algorithm consists of an initialization stage and a refinement stage.

\paragraph{Initialization: weighted $k$-medians clustering}
{To explain the rationale behind our proposal,}
with slight abuse of notation, let $P = (P_{ij})\in [0,1]^{n\times n}$, where
for all $i,j\in [n]$, $P_{ij} = P_{ji} = \theta_i\theta_j B_{z(i) z(j)}$. 
Except for the diagonal entries, $P$ is the same as in \eqref{eq:DCBM-space}.
For any $i\in [n]$, let $P_i$ denote the $i\Th$ row of $P$. 
Then for all $i$ such that $z(i)=u$, we observe that
\begin{equation*}
\theta_i^{-1}P_i=(\theta_1 B_{u,z(1)}, \dots,  \theta_n B_{u,z(n)})	
\end{equation*}
are all equal.
Thus, there are exactly $k$ different vectors that the normalized row vectors $\{\theta_i^{-1}P_i\}_{i=1}^n$ can be.
Moreover, which one of the $k$ vectors the $i\Th$ normalized row vector equals is determined solely by its community label $z(i)$.
This observation suggests one can design a reasonable community detection procedure by clustering the sample counterparts of the vectors $\{\theta_1^{-1}P_1,\theta_2^{-1}P_2,...,\theta_n^{-1}P_n\}$,
which leads us to the proposal of Algorithm \ref{alg:algl3}.

In Algorithm \ref{alg:algl3}, 
Steps 1 and 2 aim to find an estimator $\hat{P}$ of $P$ by solving a low rank approximation problem. 
Then, in Step 3, we can use $\norm{\hat{P}_i}_1^{-1}\hat{P}_i$ as a surrogate for $\theta_i^{-1}P_i$.
Finally, Step 4 performs a weighted $k$-median clustering procedure applied on the row vectors of the $n\times k$ matrix
$\begin{bmatrix}
\norm{\hat{P}_1}_1^{-1}\hat{P}_1 \\ \cdots \\ \norm{\hat{P}_n}_1^{-1}\hat{P}_n 
\end{bmatrix}$.

The main novelty of the proposed Algorithm \ref{alg:algl3} lies in the first two steps.
To improve the effect of denoising in the sparse regime, Step 1 removes the rows and the columns of $A$ whose sums are too large. This idea was previously used in community detection in SBMs \citep{chin2015stochastic}.
{If one omits this step, the high probability error bound for the output of Algorithm \ref{alg:algl3} could suffer an extra multiplier of order $O(\log n)$.}
The choice of $\tau$ will be made clear in \prettyref{lem:initial3} and \prettyref{rmk:tau} below.
Note that the potential loss of information in Step 1 for those highly important nodes will be later recovered in the refinement state.
The $\hat{P}$ matrix sought in Step 2 can be obtained by an eigen-decomposition of $T_{\tau}(A)$.
That is, $\hat{P}=\hat{U}\hat{\Lambda}\hat{U}^T$, where $\hat{U}\in\mathbb{R}^{n\times k}$ collects the $k$ leading eigenvectors, and $\hat{\Lambda}$ is a diagonal matrix of top $k$ eigenvalues.
A notable difference between Algorithm \ref{alg:algl3} and many existing spectral clustering algorithms (e.g., \cite{qin2013regularized, lei2015consistency, jin2015fast}) is that we work with the estimated probability matrix $\hat{P}$ directly rather than its leading eigenvectors $\hat{U}$.
As we shall see later, such a difference allows us to avoid eigen-gap assumption required for performance guarantees in the aforementioned papers.
{Using weighted $k$-median in step 4 is mainly for technical reasons, as it allows us to establish the same error bound under weaker conditions.}
In a recent paper \cite{chen2015convexified}, a weighted $k$-medians algorithm was also used in community detection in DCBMs. 
A key difference is that we apply it on the matrix $\hat{P}$, 
while \cite{chen2015convexified} applied it on an estimator of the 
membership
matrix $(\indc{z(i)=z(j)})\in \sth{0,1}^{n\times n}$ obtained from a convex program.

\begin{algorithm}[htb]
    \SetAlgoLined
    \KwData{Adjacency matrix $A\in\{0,1\}^{n\times n}$, number of clusters $k$, tuning parameter $\tau$.}
    \KwResult{Initial label estimator $\hat z^0$.}
    \nl Define $T_\tau(A)\in\{0,1\}^{n\times n}$ by replacing the $i$th row and column of $A$ whose row sum is larger than $\tau$ by zeroes for each $i\in[n]$\;
    \nl Solve $$\hat{P}=\argmin_{\text{rank}(P)\leq k}\fnorm{T_{\tau}(A)-P}^2;$$
    \nl Let $\hat P_i$ be the $i\Th$ row of $\hat P$. Define $S_0=\{i\in[n]:\norm{\hat P_i}_1=0\}$. Set $\hat z^0(i)=0$ for $i\in S_0$, and define $\tilde P_i=\hat P_i/\norm{\hat P_i}_1$ for $i\notin S_0$\;
    \nl Solve a $(1+\epsilon)$-$k$-median optimization problem on $S_0^c$. That is, find $\{\hat{z}^0(i)\}_{i\in S_0^c}$ in $[k]^{|S_0^c|}$ that satisfies
    \begin{equation}
    \sum_{u=1}^k\min_{v_u\in\mathbb{R}^n}\sum_{\{i\in S_0^c:\hat{z}^0(i)=u\}}\norm{\hat P_i}_1\norm{\tilde P_i-v_u}_1\leq (1+\epsilon)\min_{z\in[k]^n}\sum_{u=1}^k\min_{v_u\in\mathbb{R}^n}\sum_{\{i\in S_0^c:z(i)=u\}}\norm{\hat P_i}_1\norm{\tilde P_i-v_u}_1.\label{eq:opt-km}
    \end{equation}
\caption{Weighted $k$-medians Clustering\label{alg:algl3}}
\end{algorithm}




\paragraph{Refinement: normalized network neighbor counts}
As we shall show later, the error rate of Algorithm \ref{alg:algl3} decays polynomially with respect to the key quantity $I$ defined in \eqref{eq:formula-I}. 
To achieve the desired exponential decay rate with respect to $I$ as in the minimax rate, we need to further refine the community assignments obtained from Algorithm \prettyref{alg:algl3}.

To this end, we propose a prototypical refinement procedure in Algorithm \ref{alg:refine}.
The algorithm determines a possibly new community label for the $i\Th$ node by counting the number of neighbors that the $i\Th$ node has in each community normalized by the corresponding community size, and then picking the label of the community that maximizes the normalized counts.
If there is a tie, we break it in an arbitrary way.

\begin{algorithm}[tb]
	\label{alg:refine}
	\SetAlgoLined
	\KwData{Adjacency matrix $A\in\{0,1\}^{n\times n}$, number of clusters $k$ and a community label vector $\hat{z}^0$;}
	\KwResult{A refined community label vector $\hat{z}\in[k]^n$;}
	\caption{A Prototypical Refinement Procedure}
	\nl For each $i\in[n]$, let
	\begin{equation}
	\hat z(i)=\argmax_{u\in[k]}\frac{1}{|\{j:\hat z^0(j)=u\}|}\sum_{\{j:\hat z^0(j)=u\}}A_{ij}.\label{eq:ave-edge}
	\end{equation}
\end{algorithm}


To see the rationale behind Algorithm \ref{alg:refine}, let us consider a simplified version of the problem.
Suppose $k = 2$, $n = 2m + 1$ for some integer $m\geq 1$, $B_{11} = B_{22} = p$ and $B_{12} = B_{21} = q$.
Moreover, let us assume that the community labels of the first $2m$ nodes are such that $z(i) = 1$ for $i=1,\dots,m$ and $z(i)=2$ for $i=m+1,\dots, 2m$. The label of the last node $z(n)$ remains to be determined from the data.
When $\sth{z(i): i=1,\dots, 2m}$ are the truth, the determination of the label for the $n\Th$ node reduces to the following testing problem:
\begin{equation}
\begin{aligned}
H_0: \{A_{n,i}\}_{i\in[n-1]} &\sim\bigotimes_{i=1}^{m}\text{Bern}(\theta_{n}\theta_ip)\otimes\bigotimes_{i=m+1}^{2m}\text{Bern}(\theta_{n}\theta_iq),	\quad \mbox{vs.}\\
H_1: \{A_{n,i}\}_{i\in[n-1]}
&\sim\bigotimes_{i=1}^m\text{Bern}(\theta_{n}\theta_iq)\otimes\bigotimes_{i=m+1}^{2m}\text{Bern}(\theta_{n}\theta_ip).
\end{aligned}
\label{eq:2-test}	
\end{equation}
The hypotheses $H_0$ and $H_1$ are joint distributions of $\{A_{n,i}\}_{i\in[n-1]}$ in the two cases $z(n)=1$ and $z(n)=2$, respectively.
For this simple vs.~simple testing problem, the Neyman--Pearson lemma dictates that the likelihood ratio test is optimal.
However, it is not a satisfying answer for our goal, since the likelihood ratio test needs to use the values of the unknown parameters $p$, $q$ and $\theta$.
While it is possible to obtain sufficiently accurate estimators for $p$ and $q$, it is hard to do so for $\theta$, especially when the network is sparse.
In summary, the dependence of the likelihood ratio test on nuisance parameters makes it impossible to apply in practice.
To overcome this difficulty, we propose to consider a simple test which
\begin{equation}
\mbox{rejects $H_0$ if}
\sum_{i:z(i)=1}A_{n,i} < \sum_{i:z(i)=2}A_{n,i}.
\label{eq:count}
\end{equation}
As we shall show later in Lemma \ref{lem:t-lower} and Lemma \ref{lem:t-upper}, this simple procedure achieves the optimal testing error exponent.
It is worthwhile to point out that it does not require any knowledge of $p$, $q$ or $\theta$, and hence the procedure is adaptive. 
A detailed study of the testing problem (\ref{eq:2-test}) is given in Section \ref{sec:testing}.

Inspired by the foregoing discussion, when $k=2$ and the two community sizes are different, we propose to normalize the counts in \eqref{eq:count} by the community sizes. 
Moreover, when there are more than two communities, we propose to perform pairwise comparison based on the foregoing (normalized) test statistic for each pair of community labels, which becomes the procedure in \eqref{eq:ave-edge} as long as we replace the unknown truth $z$ with an initial estimator $\hat z^0$. For a good initial estimator such as the one output by Algorithm \ref{alg:algl3}, the refinement can lead to minimax optimal errors in misclassification proportion for a large collection of parameter spaces.

\subsection{Performance Guarantees}

In this part, we state high probability performance guarantees for the proposed procedure. The theoretical property of the algorithms requires an extra bound on the maximal entry of $\mathbb{E}A$. We incorporate this condition into the following parameter space
\begin{eqnarray*}
&& \mathcal{P}_n'(\theta,p,q,k,\beta;\delta,\alpha) \\
&& =\big\{P=(\theta_i\theta_jB_{z(i)z(j)}\indc{i\neq j})\in\mathcal{P}_n(\theta,p,q,k,\beta;\delta): \max_{u\in[k]}B_{uu}\leq\alpha p \big\}.
\end{eqnarray*}
The parameter $\alpha$ is assumed to be a constant no smaller than $1$ that does not change with $n$.
By studying the proofs of Theorem \ref{thm:lower-nbar} and Theorem \ref{thm:upper-nbar}, the minimax lower and upper bounds do not change for the slightly smaller parameter space $\mathcal{P}_n'(\theta,p,q,k,\beta;\delta,\alpha)$. Therefore, the rate $\exp(-(1+o(1))I)$ still serves as a benchmark for us to develop theoretically justifiable algorithms for the parameter space $\mathcal{P}_n'(\theta,p,q,k,\beta;\delta,\alpha)$.

\paragraph{Error rate for the initialization stage}
As a first step, we provide the following high probability error bound for Algorithm \ref{alg:algl3}.

\begin{lemma}[Error Bound for Algorithm \ref{alg:algl3}]\label{lem:initial3}
Assume $\delta=o(1)$, $1<p/q=O(1)$ and $\norm{\theta}_{\infty}=o(n/k)$. 
Let $\tau=C_1(np\norm{\theta}_{\infty}^2+1)$ for some sufficiently large constant $C_1>0$ in Algorithm \ref{alg:algl3}. 
Then, there exist some constants $C',C>0$, such that for any generative model in $\mathcal{P}_n'(\theta,p,q,k,\beta;\delta,\alpha)$, we have with probability at least $1-n^{-(1+C')}$,
$$\min_{\pi\in\Pi_k}\sum_{\{i:\hat{z}(i)\neq \pi(z(i))\}}\theta_i\leq C\frac{(1+\epsilon)k^{5/2}\sqrt{n\norm{\theta}_{\infty}^2p+1}}{p-q}\,.
$$
\end{lemma}

Lemma \ref{lem:initial3} provides a uniform high probability bound for the sum of $\theta_i$'s of the nodes which are assigned wrong labels. Before discussing the implication of this result, we give two remarks.

\begin{remark}
	\vskip -0.75em
Algorithm \ref{alg:algl3} applies a weighted $k$-medians procedure on the matrix $\hat{P}$ instead of its leading eigenvectors. This is the main difference between Algorithm \ref{alg:algl3} and many traditional spectral clustering algorithms. As a result, we avoid any eigengap assumption that is imposed to prove consistency results for spectral clustering algorithms \citep{rohe2011spectral,joseph2013impact,qin2013regularized, lei2015consistency, jin2015fast}.
\end{remark}

\begin{remark}
	\label{rmk:tau}
		\vskip -0.75em
Lemma \ref{lem:initial3} suggests that the thresholding parameter $\tau$ in Algorithm \ref{alg:algl3} should be set at the order of $np\norm{\theta}_{\infty}^2+1$. Under the extra assumption $\frac{\max_{i\neq j}\mathbb{E}A_{ij}}{\min_{i\neq j}\mathbb{E}A_{ij}}=O(1)$, we can use a data-driven version $\tau=C_1\frac{1}{n}\sum_{i\neq j}A_{ij}$ for some large constant $C_1>0$. The result of Lemma \ref{lem:initial3} stays the same.
\end{remark}

\begin{remark}
		\vskip -0.75em
The extra $(1+\epsilon)$ slack that we allow in Algorithm \ref{alg:algl3} is also reflected in the error bound. 
\end{remark}

	\vskip -0.75em
The following corollary exemplifies how the result of Lemma \ref{lem:initial3} can be specialized into a high probability bound for the loss function \eqref{eq:loss} with a rate depending on $I$ under some stronger conditions.
These conditions, especially $\min_i\theta_i=\Omega(1)$, can be relaxed in Theorem \ref{thm:algo-J} stated in the next paragraph.

\begin{corollary}
	\label{cor:ini}
Under the conditions of \prettyref{lem:initial3}, if we further assume $p\geq n^{-1}$, $k=O(1)$ and $\Omega(1)=\min_i\theta_i\leq\norm{\theta}_{\infty}=O(1)$, then there exist some constants $C',C>0$, such that for any generative model in $\mathcal{P}_n'(\theta,p,q,k,\beta;\delta,\alpha)$, we have
$\ell(\hat{z},z)\leq  C(1+\epsilon) I^{-1/2}$	
with probability at least $1-n^{-(1+C')}$.
\end{corollary}

\paragraph{Error rate for the refinement stage}

As exemplified in Corollary \ref{cor:ini}, the convergence rate for the initialization step is typically only polynomial in $I$ as opposed to the exponential rate in the minimax rate.
Thus, there is room for improvement.
In what follows, we show that a specific way of applying Algorithm \ref{alg:refine} on the output of Algorithm \ref{alg:algl3} leads to significant performance enhancement in terms of misclassification proportion.
To this end, let us first state in Algorithm \ref{alg:prove_refine} the combined algorithm for which we are able to establish the improved error bounds.
Here and after, for any $i\in [n]$, let $A_{-i}\in \{0,1\}^{(n-1)\times (n-1)}$ be the submatrix of $A$ obtained from removing the $i\Th$ row and column of $A$.

\begin{algorithm}[!tbh]\label{alg:prove_refine}
	\SetAlgoLined
	\KwData{Adjacency matrix $A\in\{0,1\}^{n\times n}$ and number of clusters $k$;}
	\KwResult{Clustering label estimator $\hat{z}\in[k]^n$;}
	\caption{A Provable Version of Algorithm \ref{alg:algl3} $+$ Algorithm \ref{alg:refine}}
	\nl For each $i\in[n]$, apply Algorithm \ref{alg:algl3} to $A_{-i}$. The result, which is a vector of dimension $n-1$, is stored in $(\hat{z}^0_{-i}(1),...,\hat{z}^0_{-i}(i-1),\hat{z}^0_{-i}(i+1),...,\hat{z}^0_{-i}(n))$\;
	\nl For each $i\in[n]$, the $i$th entry of $\hat{z}_{-i}^0$ is set as
	$$\hat{z}_{-i}^0(i)=\argmax_{u\in[k]}\frac{1}{|\{j:\hat z_{-i}^0(j)=u\}|}\sum_{j:\hat z_{-i}^0(j)=u}A_{ij};$$
	\nl Set $\hat{z}(1)=\hat{z}_{-1}^0(1)$. For each $i\in\{2,...,n\}$, set
	\begin{equation}
	\hat{z}(i)=\argmax_{u\in[k]}|\{j:\hat{z}_{-1}^0(j)=u\}\cap\{j:\hat{z}_{-i}^0(j)=\hat{z}_{-i}^0(i)\}|.\label{eq:consensus}
	\end{equation}
\end{algorithm}

\begin{remark}
		\vskip -0.75em
The last step (\ref{eq:consensus}) of Algorithm \ref{alg:prove_refine} constructs a final label estimator $\hat{z}$ from $\hat{z}_{-1}^0,\hat{z}_{-2}^0,...,\hat{z}_{-n}^0$. Since the labels given by $\hat{z}_{-1}^0,\hat{z}_{-2}^0,...,\hat{z}_{-n}^0$ are only comparable after certain permutations in $\Pi_k$, we need this extra step to resolve this issue. 
\end{remark}

\begin{remark}
	\vskip -0.75em
Algorithm \ref{alg:prove_refine} is a theoretically justifiable version for combining Algorithm \ref{alg:algl3} and Algorithm \ref{alg:refine}. In order to obtain a rate-optimal label assignment for the $i\Th$ node, we first apply the initial clustering procedure in Algorithm \ref{alg:algl3} on the sub-network consisting of the remaining $n-1$ nodes and the edges among them. Then, one applies the refinement procedure in Algorithm \ref{alg:refine} to assign a label for the $i\Th$ node. 
The independence between the initialization and the refinement stages facilitates the technical arguments in the proof. 
However, in practice, one can simply apply Algorithm \ref{alg:algl3} followed by Algorithm \ref{alg:refine}. 
The numerical difference from Algorithm \ref{alg:prove_refine} is negligible in all the data examples we have examined.
\end{remark}

\subparagraph{A special case: almost equal community sizes}
In the special case where the community sizes are almost equal, we can show that $\hat{z}$ output by Algorithm \ref{alg:prove_refine} achieves the minimax rate. 
\begin{theorem}
	\label{thm:eq-size-I}
Under the conditions of \prettyref{lem:initial3},
we further assume $\beta = 1$, $k=O(1)$, $\min_i\theta_i = \Omega(1)$, $\delta = o(\frac{p-q}{p})$,  $\norm{\theta}_\infty^2 p\geq n^{-1}$, and $\frac{(1+\epsilon)\norm{\theta}_{\infty} p^{3/2}}{\sqrt{n}(p-q)^2}=o(1)$.
Then there is a sequence $\eta = o(1)$ such that the output $\hat{z}$ of Algorithm \ref{alg:prove_refine} satisfies
\begin{equation*}
\lim_{n\to\infty} \inf_{\calP_n'(\theta,p,q,k,\beta;\delta,\alpha)} \Prob\sth{\ell(\hat{z},z)\leq \exp\big(-(1-\eta)I\big)} = 1.
\end{equation*}
\end{theorem}

Theorem \ref{thm:eq-size-I} shows that when the community sizes are almost equal, the minimax rate $\exp(-(1+o(1))I)$ can be achieved within polynomial time. We note that the conditions that we need here are stronger than those of Theorem \ref{thm:upper-nbar}. When $k=O(1)$, $\epsilon=O(1)$ and $\Omega(1)=\min_i\theta_i\leq\norm{\theta}_{\infty}=O(1)$, Theorem \ref{thm:upper-nbar} only requires $I\rightarrow\infty$, which is equivalent to $\frac{p^{1/2}}{\sqrt{n}(p-q)}=o(1)$, while Theorem \ref{thm:eq-size-I} requires $\frac{p^{3/2}}{\sqrt{n}(p-q)^2}=o(1)$. Whether the extra factor $\frac{p}{p-q}$ can be removed from the assumptions is an interesting problem to investigate in the future.

\subparagraph{General case}
We now state a general high probability error bound for Algorithm \ref{alg:prove_refine}.
To introduce this result, we define another information-theoretic quantity. For any $t\in(0,1)$, define
\begin{equation}
	\label{eq:Jtpq}
J_t(p,q)=2\left(tp+(1-t)q-p^tq^{1-t}\right).	
\end{equation}
By Jensen's inequality, it is straightforward to verify that $J_t(p,q)\geq 0$ and $J_t(p,q)=0$ if and only if $p=q$. As a special case, when $t=\frac{1}{2}$, we have 
\begin{equation}
	\label{eq:J-hf}
J_{1\over 2}(p,q)=(\sqrt{p}-\sqrt{q})^2.	
\end{equation}
For a given $z\in[k]^n$, let $n_{(1)}\leq ...\leq n_{(k)}$ be the order statistics of community sizes $\{n_{u}(z): u=1,\dots,k\}$. 
Then, we define the quantity $J$ by through
\begin{equation}
\exp(-J)=\frac{1}{n}\sum_{i=1}^n\exp\left(-\theta_i\,\left(\frac{n_{(1)}+n_{(2)}}{2}\right)\,J_{t^*}(p,q)\right)
\label{eq:J-def}
\end{equation}
with $t^* = \frac{n_{(1)}}{n_{(1)}+n_{(2)}}$.
With the foregoing definitions, the following theorem gives a general error bound for Algorithm \ref{alg:prove_refine}.

\begin{theorem}
	\label{thm:algo-J}
Under the conditions of Lemma \ref{lem:initial3}, we further assume
 that $\delta = o(\frac{p-q}{p})$, $\norm{\theta}_{\infty}^2p\geq n^{-1}$,
\begin{align}
\label{eqn:algo-J1}
\frac{(1+\epsilon)k^{5/2}\norm{\theta}_{\infty}\sqrt{p}}{\sqrt{n}(p-q)}=o\left(\frac{p-q}{kp}\right), \quad \mbox{and}\\
\label{eqn:algo-J2}
\min_{\gamma\geq 0}\left\{n^{-1}|\{i\in[n]:\theta_i\leq\gamma\}|+\frac{(1+\epsilon)k^{5/2}\norm{\theta}_{\infty}\sqrt{p}}{\gamma\sqrt{n}(p-q)}\right\}=o\left(\frac{p-q}{k^2p}\right).
\end{align}
Then there is a sequence $\eta = o(1)$ such that the output $\hat{z}$ of Algorithm \ref{alg:prove_refine} satisfies
\begin{equation*}
\lim_{n\to\infty} \inf_{\mathcal{P}_n'(\theta,p,q,k,\beta;\delta,\alpha)} \Prob\sth{\ell(\hat{z},z)\leq \exp\big(-(1-\eta)J\big)} = 1.
\end{equation*}
%
\end{theorem}

Theorem \ref{thm:algo-J} gives a general error bound for the performance of Algorithm \ref{alg:prove_refine}. It is easy to check that the conditions (\ref{eqn:algo-J1}) and (\ref{eqn:algo-J2}) are satisfied under the settings of Theorem \ref{thm:eq-size-I}. Therefore, Theorem \ref{thm:eq-size-I} is a special case of Theorem \ref{thm:algo-J}.
Theorem \ref{thm:algo-J} shows that Algorithm \ref{alg:prove_refine} converges at the rate $\exp(-(1+o(1))J)$. According to the properties of $J_t(p,q)$ stated in Appendix \ref{sec:property}, one can show that when $n_{(1)} = (1+o(1))n_{(2)}$, $J = (1+o(1))I$, and that in general
$$n_{(1)}(\sqrt{p}-\sqrt{q})^2\leq\left(\frac{n_{(1)}+n_{(2)}}{2}\right)\,J_{t^*}(p,q)\leq \left(\frac{n_{(1)}+n_{(2)}}{2}\right)(\sqrt{p}-\sqrt{q})^2.$$
Using this relation, we can state the convergence rate in Theorem \ref{thm:algo-J} using the quantity $I$.

\begin{corollary}
\label{coro:I-J}
Under the conditions of \prettyref{thm:algo-J}, 
there is a sequence $\eta= o(1)$ such that the output $\hat{z}$ of Algorithm \ref{alg:prove_refine} satisfies
\begin{align*}
\lim_{n\to\infty} & \inf_{\mathcal{P}_n'(\theta,p,q,2,\beta;\delta,\alpha)} \Prob\sth{\ell(\hat{z},z) \leq \exp\big(-(1-\eta) \beta^{-1}I\big)} = 1, \\
\lim_{n\to\infty} & \inf_{\mathcal{P}_n'(\theta,p,q,k,\beta;\delta,\alpha)} \Prob\sth{\ell(\hat{z},z) \leq \exp\big(-(1-\eta)I\big)} = 1, \text{ for }k\geq 3.
\end{align*}
\end{corollary}

Therefore, when $k\geq 3$, the minimax rate $\exp(-(1+o(1))I)$ is achieved by Algorithm \ref{alg:prove_refine}. The only situation where the minimax rate is not achieved by Algorithm \ref{alg:prove_refine} is when $k=2$ and $\beta>1$. For this case, there is an extra $\beta^{-1}$ factor on the exponent of the convergence rate.

\begin{remark}
		\vskip -0.75em
If we further assume that $\min_{i\neq j} B_{ij} = \Omega(q)$, a careful examination of the proofs shows that we can improve the term $k^{5/2}$ in the conclusion of \prettyref{lem:initial3} and in the conditions \eqref{eqn:algo-J1} and \eqref{eqn:algo-J2} to $k^{3/2}$.
Since it is unclear what the optimal power exponent for $k$ is in these circumstances, we do not pursue it explicitly in this paper.	
\end{remark}

\section{Numerical Results}
\label{sec:nume}

In this section,
we present numerical experiments on simulated datasets generated from DCBMs. 
In particular, 
we compare the performance of two versions of our algorithm with two state-of-the-art methods: SCORE \citep{jin2015fast} and CMM \citep{chen2015convexified} in two different scenarios. 
On simulated data examples, both versions of our algorithm outperformed SCORE in terms of misclassification proportion. 
The performance of CMM was comparable to our algorithm.
However, our algorithm not only demonstrated slightly better accuracy on the simulated datasets when compared with CMM, but it also enjoys the advantage of easy implementation, fast computation and scalability to networks of large sizes since it does not involve convex programming.


\paragraph{Scenario 1}
We set $n=300$ nodes and $k=2$. 
The sizes of the two communities were set as $100$ and $200$, respectively. 
The off-diagonal entries of the adjacency matrix were generated as $A_{ij} = A_{ji}\stackrel{ind.}{\sim}\text{Bern}(\theta_i\theta_j p)$ if $z(i)=z(j)$ and $A_{ij}= A_{ji}\stackrel{ind.}{\sim}\text{Bern}(\theta_i\theta_jq)$ if $z(i)\neq z(j)$. 
We let $p=0.1$ and $q=3p/10$. 
The degree-correction parameters were set as $\theta_i = |Z_i|+1-(2\pi)^{-1/2}$ where $Z_i\stackrel{iid}{\sim} N(0,0.25)$ for $i=1,\dots,n$. 
It is straightforward to verify that $\Expect \theta_i=1$.

We compare misclassification proportions of the following five algorithms\footnote{The numerical performance of Algorithm \ref{alg:prove_refine} was indistinguishable from that of the second algorithm in the list in all the experiments conducted.}:
\begin{enumerate}
\item The weighted $k$-medians procedure in Algorithm \ref{alg:algl3};
\item Refinement of the output of Algorithm \ref{alg:algl3} by Algorithm \ref{alg:refine};
\item Iterate Algorithm \ref{alg:refine} $10$ times after initialization by Algorithm \ref{alg:algl3}.
\item The SCORE method in \cite{jin2015fast};
\item The CMM method in \cite{chen2015convexified}.
\end{enumerate}

\begin{figure}[h]
\centering
\includegraphics[width=0.58\textwidth]{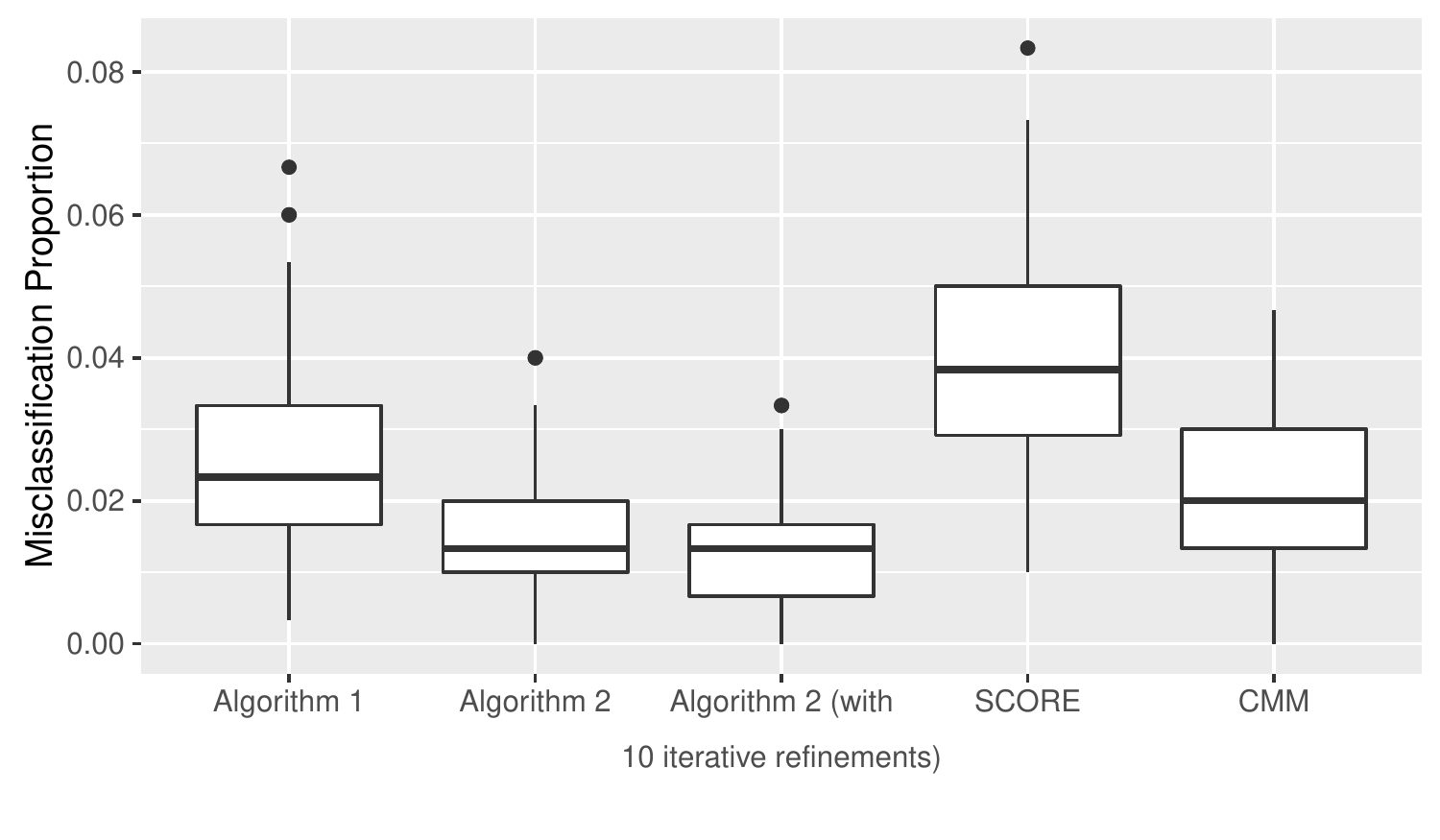}
\includegraphics[width=0.4\textwidth]{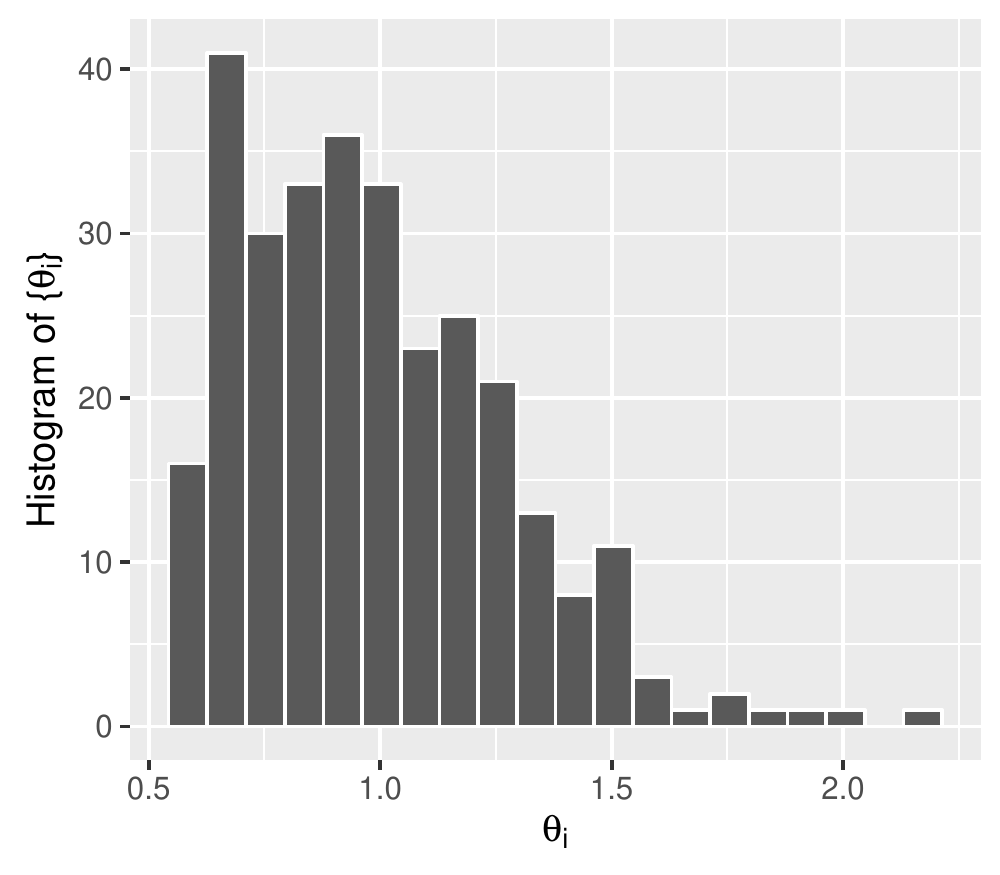}
\caption{Left panel: boxplots of misclassification proportions for the five algorithms over $100$ independent repetitions. Right panel: histogram of $\theta_i$. \label{fig:num1}}
\end{figure}

We conducted the experiments with $100$ independent repetitions and summarize the performances of the five algorithms through boxplots of misclassification proportions.
\prettyref{fig:num1} shows that our refinement step (Algorithm \ref{alg:refine}) significantly improved the performance of the initialization step (Algorithm \ref{alg:algl3}). 
Moreover, it helped to further reduce the error if we apply the refinement step for a few more iterations. 
Among the five algorithms, our proposed procedures give the best performance.
The CMM algorithm performed slightly worse than our procedures with refinement, but was better than Algorithm \ref{alg:algl3} and SCORE.




\paragraph{Scenario 2}

Here, we set $n=800$ and $k=4$ and all community sizes were set equal. 
The adjacency matrix was generated in the same way as in Scenario 1 except that
$\theta_i$'s were independent draws from a Pareto distribution with density function $f(x)=\frac{\alpha \beta^\alpha}{x^{\alpha+1}}1_{\{x\geq \beta\}}$, with $\alpha=5$ and $\beta=4/5$. 
The choice of $\alpha$ and $\beta$ ensures that $\mathbb{E}\theta_i=1$.

\begin{figure}[h]
\centering
\includegraphics[width=0.58\textwidth]{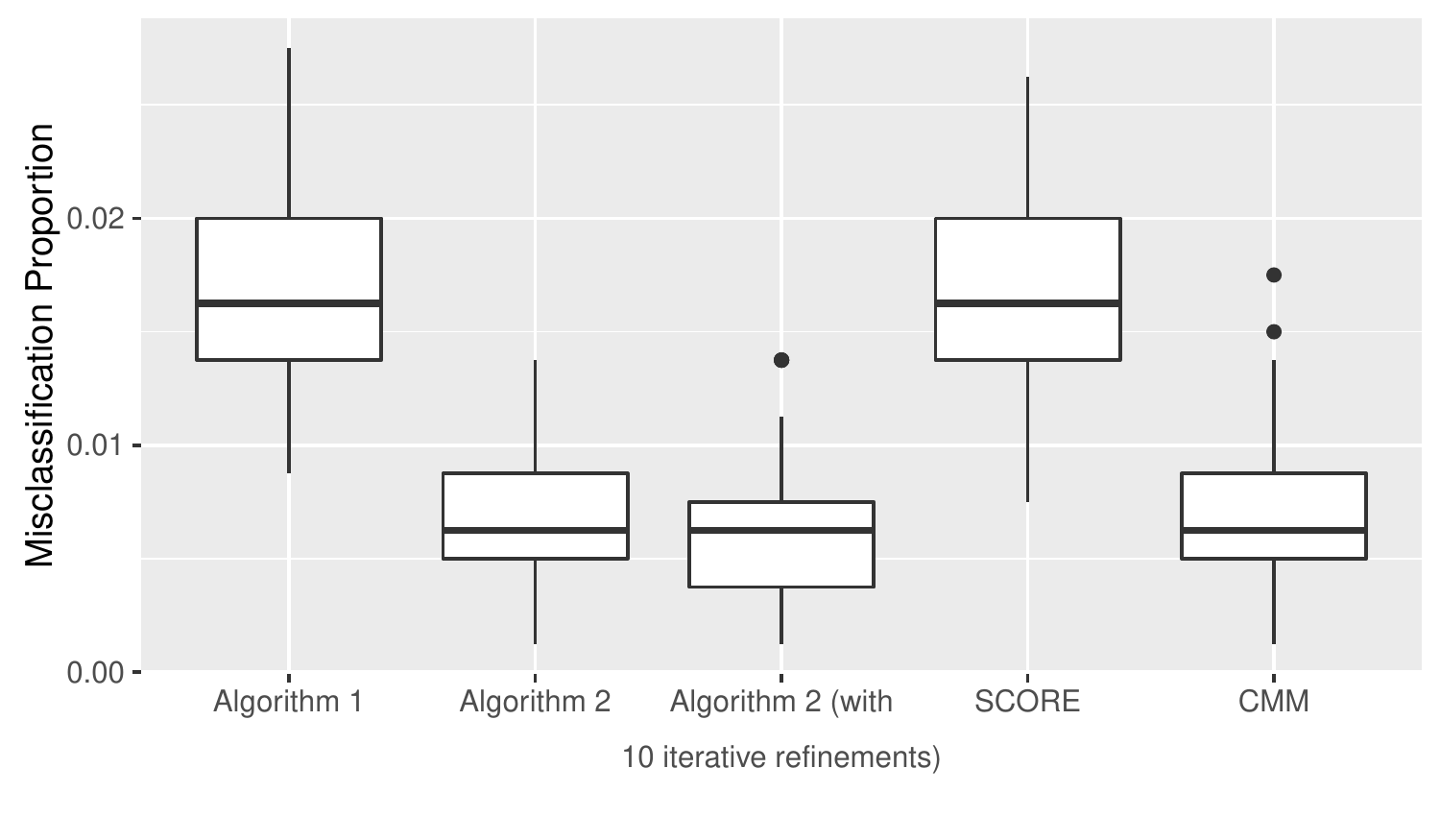}
\includegraphics[width=0.4\textwidth]{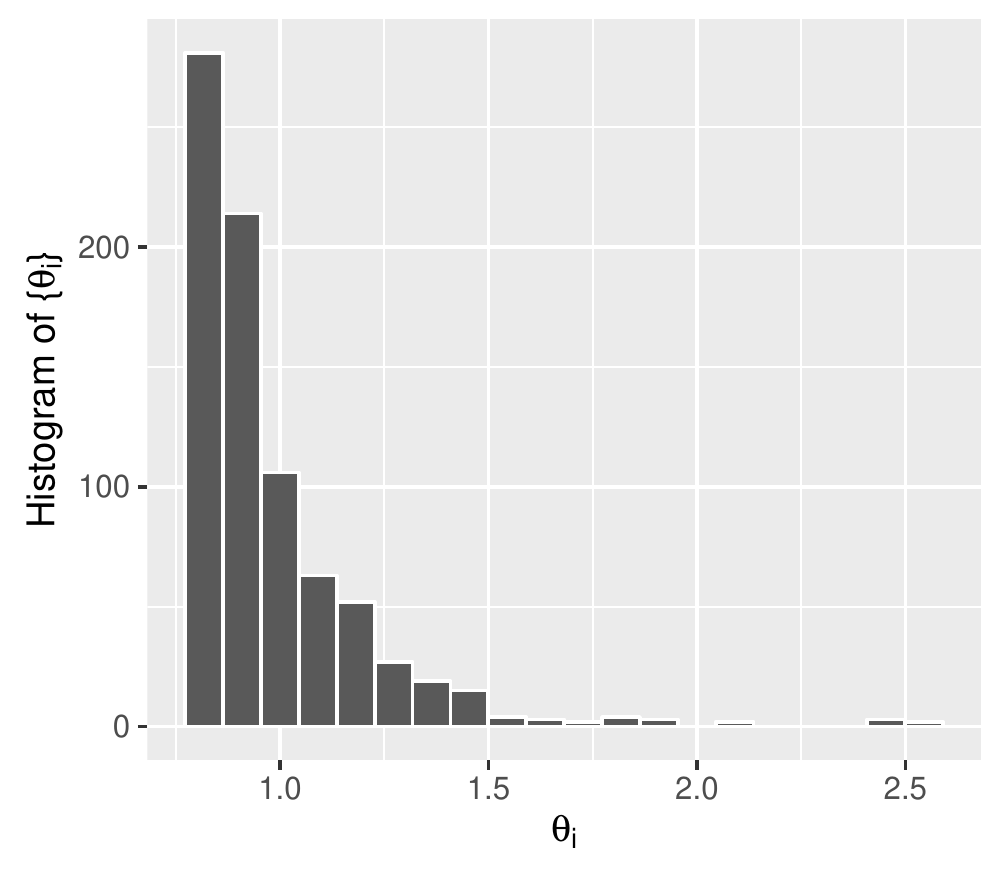}
\caption{Left panel: boxplots of misclassification proportions for the five algorithms over $100$ independent repetitions. Right panel: histogram of $\theta_i$.\label{fig:num2}}
\end{figure}

As in Scenario 1, we compare the performance of the five algorithms over $100$ independent repetitions. \prettyref{fig:num2} shows the boxplots of the misclassification proportions. 
The overall message is similar to Scenario 1, except that CMM performed almost as good as our procedures with refinement, but all of them outperformed Algorithm \ref{alg:algl3} and SCORE.
Despite its comparable performance on this task, CMM required noticeably longer running time than our procedures on the simulated datasets due to the involvement of convex programming.
Therefore, its scalability to large networks is more limited.

\section{Proofs}
\label{sec:proof}

This section presents proofs for some main results of the paper.
In Section \ref{sec:testing}, we first study a fundamental testing problem for community detection in DCBMs. 
The theoretical results for this testing problem are critical tools to prove minimax lower and upper bounds for community detection.
We then state the proof of \prettyref{thm:upper-nbar}. 
Proofs of the other main results are deferred to the appendices.



\subsection{A Fundamental Testing Problem}
\label{sec:testing}

As a prelude to all proofs, we consider the hypothesis testing problem \eqref{eq:2-test} that not only is fundamental to the study of minimax risk but also motivates the proposal of Algorithm \ref{alg:refine}.
To paraphrase the problem,
suppose $X = (X_1,\dots, X_{m},X_{m+1},\dots, X_{2m})$ have independent Bernoulli entries. 
Given $1 \geq p>q \geq 0$ and $\theta_0,\theta_1,\dots, \theta_{2m} > 0$ such that
\begin{equation}
	\label{eq:normalize-test}
\sum_{i=1}^{m}\theta_i=\sum_{i=m+1}^{2m}\theta_i=m.
\end{equation}
We are interested in understanding the minimum possible Type I+II error of testing
\begin{equation}
\label{eq:hypo}
\begin{aligned}
& H_0: X\sim \bigotimes_{i=1}^{m}\text{Bern}\left(\theta_0\theta_ip\right) \otimes \bigotimes_{i=m+1}^{2m}\text{Bern}\left(\theta_0\theta_iq\right) \\
& \quad \quad \quad \mbox{vs.}\quad
H_1: X\sim \bigotimes_{i=1}^{m}\text{Bern}\left(\theta_0\theta_i q\right) \otimes \bigotimes_{i=m+1}^{2m}\text{Bern}\left(\theta_0\theta_i p\right).	
\end{aligned}
\end{equation}
In particular, we are interested in the asymptotic behavior of the error for a sequence of such testing problems in which $p$, $q$ and the $\theta_i$'s scale with $m$ as $m\to\infty$.
First, we have the following lower bound result.
\begin{lemma}
	\label{lem:t-lower}
Suppose that as $m\to\infty$, $1<p/q = O(1)$ and $p \max_{0\leq i\leq 2m}\theta_i^2 = o(1)$. 
If $\theta_0 m (\sqrt{p}-\sqrt{q})^2 \to \infty$, 
\begin{equation*}
\inf_{\phi}\left(P_{H_0}\phi+P_{H_1}(1-\phi)\right)\geq 
\exp\left(-(1+o(1)) \theta_0 m (\sqrt{p}-\sqrt{q})^2\right).	
\end{equation*}
Otherwise if $\theta_0 m (\sqrt{p}-\sqrt{q})^2=O(1)$, there exists a constant $c\in (0,1)$ such that
\begin{equation*}
\inf_{\phi}\left(P_{H_0}\phi+P_{H_1}(1-\phi)\right)\geq c.
\end{equation*}
\end{lemma}

According to Neyman-Pearson lemma, the optimal testing procedure is the likelihood ratio test. However, such a test depends on the values of $\{\theta_i\}_{i=1}^{2m},p,q$, and is not appropriate in practice. We consider an alternative test
\begin{equation}
\phi=\indc{\sum_{i=1}^{m}X_i<\sum_{i=m+1}^{2m}X_i}.\label{eq:totti}
\end{equation}
This test is simple, but achieves the optimal error bound.
\begin{lemma}\label{lem:t-upper}
For the testing function defined above, we have
$$P_{H_0}\phi+P_{H_1}(1-\phi)\leq 2\exp\left(-\theta_0m(\sqrt{p}-\sqrt{q})^2\right).$$
\end{lemma}

Combining Lemma \ref{lem:t-lower} and Lemma \ref{lem:t-upper}, we find that the minimax testing error for the problem (\ref{eq:hypo}) is $e^{-(1+o(1))\theta_0m(\sqrt{p}-\sqrt{q})^2}$. This explains why the minimax rate for community detection in DCBM takes the form of $e^{-(1+o(1))I}$ in Corollary \ref{cor:minimax}. Moreover, the simple testing function (\ref{eq:totti}) serves as a critical component in Algorithm \ref{alg:refine}. The fact that (\ref{eq:totti}) can achieve the optimal testing error exponent in Lemma \ref{lem:t-upper} explains why our algorithm for community detection can achieve the minimax rate when the community sizes are equal (Theorem \ref{thm:eq-size-I}).

In order for Lemma \ref{lem:t-lower} to be applied to lower bounding the performance of community detection in DCBM, we need a version of Lemma \ref{lem:t-lower} that can handle approximately equal sizes. To be specific, 
suppose $X = (X_1,\dots, X_{m},\\ X_{m+1},\dots, X_{m+m_1})$ have independent Bernoulli entries. 
Given $1 \geq p>q \geq 0$ and $\theta_0,\theta_1,\dots, \theta_{m+m_1} > 0$ such that
\begin{equation}
	\label{eq:normalize-test-approx}
\frac{1}{m}\sum_{i=1}^{m}\theta_i, \frac{1}{m_1}\sum_{i=m+1}^{m+m_1}\theta_i\in [1-\delta,1+\delta].
\end{equation}
When $m$ and $m_1$ are approximately equal,
we are interested in understanding the minimum possible Type I+II error of testing
\begin{equation}
\label{eq:hypo-approx}
\begin{aligned}
& H_0: X\sim \bigotimes_{i=1}^{m}\text{Bern}\left(\theta_0\theta_ip\right) \otimes \bigotimes_{i=m+1}^{m+m_1}\text{Bern}\left(\theta_0\theta_iq\right) \\
& \quad \quad \quad \mbox{vs.}\quad
H_1: X\sim \bigotimes_{i=1}^{m}\text{Bern}\left(\theta_0\theta_i q\right) \otimes \bigotimes_{i=m+1}^{m+m_1}\text{Bern}\left(\theta_0\theta_i p\right).	
\end{aligned}
\end{equation}
The setting of Lemma \ref{lem:t-lower} is a special case where $m=m_1$ and $\delta=0$.

\begin{lemma}\label{lem:t-lower-aprox}
Suppose that as $m\to\infty$, $1<p/q = O(1)$, $p \max_{0\leq i\leq 2m}\theta_i^2 = o(1)$, $\delta=o(1)$ and $\left|\frac{m}{m_1}-1\right|=o(1)$. 
If $\theta_0 m (\sqrt{p}-\sqrt{q})^2 \to \infty$, 
\begin{equation*}
\inf_{\phi}\left(P_{H_0}\phi+P_{H_1}(1-\phi)\right)\geq 
\exp\left(-(1+o(1)) \theta_0 m (\sqrt{p}-\sqrt{q})^2\right).	
\end{equation*}
Otherwise, there exists a constant $c\in (0,1)$ such that
\begin{equation*}
\inf_{\phi}\left(P_{H_0}\phi+P_{H_1}(1-\phi)\right)\geq c.
\end{equation*}
\end{lemma}



\subsection{Proof of \prettyref{thm:upper-nbar}}

	\label{sec:messi}
Throughout the proof, we let $z$ denote the truth, $\hat{z}$ the estimator defined in \eqref{eq:mle} and $\tz$ a generic assignment vector.
In addition, we let $L$ denote the objective function in \eqref{eq:mle}.
In what follows, we focus on proving the upper bounds for $k\geq 3$ while the case of $k=2$ is deferred to Appendix \ref{sec:minimax-2}.

\paragraph{Outline and additional notation}
We have the following basic equality
\begin{align}
	\label{eq:basic-mle}
\mathbb{E}n\ell(\hat{z},z)=
\sum_{m=1}^n m\mathbb{P}(n\ell(\hat{z},z)=m).
\end{align}
Thus, to prove the desired upper bounds, we are to work out appropriate bounds for the individual probabilities $\mathbb{P}(n\ell(\hat{z},z)=m)$.
To this end, for any given $m$, our basic idea is to first bound $\Prob(L(\tilde{z}) > L(z))$ for any $\tilde{z}$ such that $n\ell(\hat{z},z) = m$ and then apply the union bound.
To carry out these calculations in details, we divide the entire proof into three major steps:
\begin{itemize}
\item In the first step, we derive a generic upper bound expression for the quantity  $\Prob(L(\tilde{z}) > L(z))$ for any deterministic $\tilde{z}$.

\item In the second step, we further materialize the upper bound expression according to different values of $m$ where $m = n\ell(\tilde{z},z)$.
In particular, we shall use different arguments in three different regimes of $m$ values. Together with the union bound, we shall obtain bounds for all probabilities $\mathbb{P}(n\ell(\hat{z},z)=m)$. 

\item In the last step, we supply the bounds obtained in the second step to \eqref{eq:basic-mle} to establish the theorem. 
Indeed, the bounds we derive in the second step decay geometrically once $m$ is larger than some critical value which depends on the rate of the error bounds.
Thus, we divide the final arguments here according to three different regimes of error rates.
\end{itemize}

After a brief introduction to some additional notation, we carry out these three steps in order.
We denote $n_{\min}=\min_{u\in [k]}|\{i:z(i)=u\}|$, $n_{\max}=\max_{u\in [k]}|\{i:z(i)=u\}|$ and $\theta_{\min}=\min_{i\in[n]}\theta_i$. Note that $\nmin\geq n/(\beta k)$, $\nmax\leq \beta n/k$ and $\theta_{\min}= \Omega(1)$. 
For any $t\in(0,1)$, we define
\begin{equation}\label{eq:Rt}
R_{t}=\frac{1}{n}\sum_{i=1}^n\exp\left(-(1-t)\theta_i\nmin(\sqrt{p}-\sqrt{q})^2\right).	
\end{equation}
In order to show $\mathbb{E}\ell(\hat{z},\ztrue)\leq \exp(-(1-o(1))I)$, it is sufficient to prove $\mathbb{E}\ell(\hat{z},\ztrue)\leq R_t$ for some $t=o(1)$, since 
$$R_{t}\leq \frac{1}{n}\sum_{i=1}^n\exp\left(-(1-t)\theta_in/(\beta k)(\sqrt{p}-\sqrt{q})^2\right)\leq \exp(-(1-t)I),$$ where the second inequality is by Jensen inequality.

\paragraph{Step 1: bounding $\mathbb{P}\left(L(\tz)>L(\ztrue)\right)$}
For any deterministic $\tz$, we have
\begin{align*}
\mathbb{P}\left(L(\tz)>L(\ztrue)\right)&=\p\Bigg(\sum_{\substack{i<j,\ztrue(i)=\ztrue(j)\\ \tz(i)\neq \tz(j)}}\Big(A_{ij}\log\frac{q(1-\theta_i\theta_jp)}{p(1-\theta_i\theta_jq)}+\log\frac{1-\theta_i\theta_jq}{1-\theta_i\theta_jp}\Big)\\
&~~~~~~~~+\sum_{\substack{i<j,\ztrue(i)\neq \ztrue(j)\\ \tz(i)= \tz(j)}}\Big(A_{ij}\log\frac{p(1-\theta_i\theta_jq)}{q(1-\theta_i\theta_jp)}+\log\frac{1-\theta_i\theta_jp}{1-\theta_i\theta_jq}\Big)>0\Bigg).
\end{align*}
When $\ztrue(i)=\ztrue(j)$, we have $\p(A_{ij}=1)=\theta_i\theta_jp'$ for some $p'\geq p$, and so
\begin{align*}
&
\E \exp\bigg(\frac{1}{2}\Big(A_{ij}\log\frac{q(1-\theta_i\theta_jp)}{p(1-\theta_i\theta_jq)}+\log\frac{1-\theta_i\theta_jq}{1-\theta_i\theta_jp}\Big)\bigg) \\
&=\theta_i\theta_j\sqrt{\frac{q}{p}}p'+\sqrt{\frac{1-\theta_i\theta_jq}{1-\theta_i\theta_jp}}(1-\theta_i\theta_jp')\\
&=\theta_i\theta_j\sqrt{qp}+\sqrt{(1-\theta_i\theta_jq)(1-\theta_i\theta_jp)}+\theta_i\theta_j(p'-p)\left(\sqrt{\frac{q}{p}}-\sqrt{\frac{1-\theta_i\theta_j q}{1-\theta_i\theta_jp}}\right)\\
&\leq  \exp\big(\log\Big(\sqrt{pq}\theta_i\theta_j + \sqrt{1-\theta_i\theta_jq}\sqrt{1-\theta_i\theta_jp}\big)\Big) 
\leq 
\exp(-\frac{1}{2}\theta_i\theta_j(\sqrt{q}-\sqrt{p})^2).
\end{align*}
Similarly, when $\ztrue(i)\neq \ztrue(j)$ we also have
\begin{align*}
&
\E\exp\left(\frac{1}{2}\Big(A_{ij}\log\frac{p(1-\theta_i\theta_jq)}{q(1-\theta_i\theta_jp)}+\log\frac{1-\theta_i\theta_jp}{1-\theta_i\theta_jq}\Big)\right) \\
&\leq  \exp\big(\log\Big(\sqrt{pq}\theta_i\theta_j + \sqrt{1-\theta_i\theta_jq}\sqrt{1-\theta_i\theta_jp}\big)\Big) 
\leq \exp(-\frac{1}{2}\theta_i\theta_j(\sqrt{q}-\sqrt{p})^2).
\end{align*}
For any assignment vector $\tz$,
Define membership matrix $\tilde{Y}\in\sth{0,1}^{n\times n}$ with $\tilde{Y}_{ij}= \indc{\tz(i)=\tz(j)}$ for all $i\neq j$ and zero otherwise. 
Let $Y$ be the membership matrix associated with the truth $z$.  
Note that the membership matrix is invariant under permutation of the community labels.  By applying the Chernoff bound with $t=\frac{1}{2}$ we have
\begin{equation}
\label{eq:mle-1}
\begin{aligned}
\mathbb{P}\left(L(\tz)>L(\ztrue)\right)
& \leq \prod_{\substack{i<j\\ \tilde{Y}_{ij}\neq Y_{ij}}}\exp\left(-\frac{1}{2}\theta_i\theta_j(\sqrt{p}-\sqrt{q})^2\right)\\
& = \prod_{\substack{i\neq j\\ \tilde{Y}_{ij}\neq Y_{ij}}}\exp\left(-\frac{1}{4}\theta_i\theta_j(\sqrt{p}-\sqrt{q})^2\right).
\end{aligned}
\end{equation}

\paragraph{Step 2: bounding $\Prob(n\ell(\hat{z},z) = m)$}
To obtain the desired bounds on these probabilities, we introduce a way to partition each community according to the values of the degree-correction parameters.
Given the truth $z$ and any deterministic assignment vector $\tz$,
let $\mathcal{C}_u=\{i\in[n]:\ztrue(i)=u\}$, $\Gamma_u=\{i\in[n]:\ztrue(i)=u,\tz(i)\neq u\}$ and $\Gamma = \cup_{u\in [k]} \Gamma_u$. 
Note that $\Gamma_u$ and $\Gamma$ depend on $\tilde{z}$.

Let $\folds \geq 2$ be a large enough constant integer to be determined later.
For each $\mathcal{C}_u$ we decompose it as 
$\mathcal{C}_u=\mathcal{C}_u^+\cup \mathcal{C}_u^-$ such that 
\begin{equation}
	\label{eq:Cu-decomp}
\mathcal{C}_u^+\cap \mathcal{C}_u^-=\emptyset,\quad
|\mathcal{C}_u^-| = \ceil{\frac{|\mathcal{C}_u|}{\folds}},\quad	
\min_{i\in \mathcal{C}_u^+}\theta_i\geq \max_{i\in \mathcal{C}_u^-}\theta_i.
\end{equation}
Due to the approximate normalization of degree-correction parameters, 
for sufficiently large values of $n$,
\begin{equation}
	\label{eq:Cu-}
\max_{i\in \mathcal{C}_u^-}\theta_i\leq 
{3}/{2}.	
\end{equation}
Since $|\mathcal{C}_u^+| \leq (\folds - 1) |\mathcal{C}_u^-|$, we can define a mapping $\tau_u : \mathcal{C}_u\rightarrow \mathcal{C}_u^-$ such that its restriction on $\mathcal{C}_u^-$ is identity.
Moreover, we could require that for any $i\in \mathcal{C}_u^-$, $|\tau_u^{-1}(i)|\leq \folds$.
Let $\tau$ be the mapping from $[n]$ to $\cup_{u=1}^k \mathcal{C}_u^-$ such that the restriction of $\tau$ on ${\mathcal{C}_u}$ is $\tau_u$.
{The main reason for introducing $\tau$ is to deal with the range of values the degree-correction parameters can take. 
The right side of \eqref{eq:mle-1} shows that the desired bounds depend crucially on quantities of the form $\sum_{i\in S}\theta_i$ for some set $S$.
For any set $S$, the sum $\sum_{i\in S}\theta_i$ is not necessarily upper bounded by a constant multiple of the size of the set $|S|$.
However, by \eqref{eq:Cu-}, we can always upper bound $\sum_{i\in S}\theta_{\tau(i)}$ by a constant multiple of $|S|$. 
This gives us a way to relate the probability bounds and the number of misclassified nodes.
Such a point can be seen more clearly as we go to explicit calculation below.
}

Let 
\begin{equation}
	\label{eq:m-prime}
m' = \eta n / k	
\end{equation}
for some $\eta = o(1)$ with $\eta^{-1} = o(I)$ and $k\leq n^\eta$. 
We now derive bounds for $\mathbb{P}(n\ell(\hat{z},\ztrue)=m)$ for $m\in [1,M]$, $(M,m']$ and $(m',n]$ separately.

\smallskip
\subparagraph{Case 1: $1\leq m \leq \folds$}
In this case, we have
\begin{eqnarray*}
\mathbb{P}(n\ell(\hat{z},\ztrue)=m) 
&\leq& \sum_{\tilde{z}:|\Gamma|=m}\exp\Big(-\frac{1}{2}\sum_{i\in \Gamma}\theta_i\Big((1-\delta)2\nmin-\sum_{i\in \Gamma}\theta_i\Big)(\sqrt{p}-\sqrt{q})^2\Big) \\
&\leq& \sum_{\tilde{z}:|\Gamma|=m}\exp\Big(-\frac{1}{2}\sum_{i\in \Gamma}\theta_{\tau(i)}\Big((1-\delta)2\nmin-\sum_{i\in \Gamma}\theta_{\tau(i)}\Big)(\sqrt{p}-\sqrt{q})^2\Big) \\
&\leq& \sum_{\tilde{z}:|\Gamma|=m}\prod_{i\in \Gamma}\exp\Big(-\frac{1}{2}\theta_{\tau(i)}\left((1-\delta)2\nmin-2\folds\right)(\sqrt{p}-\sqrt{q})^2\Big).
\end{eqnarray*}
Here, the first inequality comes from direct application of \eqref{eq:mle-1} and the union bound.
Since $\|\theta\|_\infty = o(n/k) = o(\nmin)$ 
and $M$ is a constant, we have $\sum_{i\in \Gamma}\theta_i = o(\nmin)$.
This, together with the monotonicity of the function $x(1-x)$ when $x$ is in the right neighborhood of zero, implies the second inequality.
The third inequality is due to \eqref{eq:Cu-}.
Since $\folds$ is a constant and $M/\nmin$ can be upper bounded by $\eta$ for large values of $n$, we further have
\begin{eqnarray*}
\mathbb{P}(n\ell(\hat{z},\ztrue)=m) 
&\leq& \sum_{\tilde{z}:|\Gamma|=m}\prod_{i\in \Gamma}\exp\Big(-\theta_{\tau(i)}(1-\delta-2\eta)2\nmin(\sqrt{p}-\sqrt{q})^2\Big) 
\\
&\leq& k^m\bigg(\sum_{i=1}^n\exp\Big(-\theta_{\tau(i)}(1-\delta-2\eta)\nmin(\sqrt{p}-\sqrt{q})^2\Big)\bigg)^m \\
&\leq& k^m\bigg(M\sum_{i=1}^n\exp\Big(-\theta_{i}(1-\delta-2\eta)\nmin(\sqrt{p}-\sqrt{q})^2\Big)\bigg)^m \\
&=& (knMR_{\delta+2\eta})^m.
\end{eqnarray*}
{Here and after, the notation $\sum_{\tilde{z}:|\Gamma|=m}$ means summing over all deterministic assignment vectors $\tilde{z}$ such that $|\Gamma| = n\ell(\tilde{z},z) = m$.}
The last inequality holds since for any $i\in \mathcal{C}_u^-$, $|\tau_u^{-1}(i)|\leq \folds$.

\smallskip
\subparagraph{Case 2: $M< m\leq m'$} 
{In this case, we cannot directly apply the argument in case 1 since we can no longer guarantee that $\sum_{i\in \Gamma}\theta_i = o(n_{\min})$ and so the second inequality of the last display no longer holds.}
To proceed,
we can further bound the rightmost side of \eqref{eq:mle-1} by $B_1\times B_2$, where
\begin{align*}
B_1 & = \prod_{\substack{(i,j):\ztrue(i)=\ztrue(j)\\ \tz(i)\neq \tz(j)}}\exp\left(-\frac{1}{4}\theta_i\theta_j(\sqrt{p}-\sqrt{q})^2\right),\\
B_2 & = \prod_{\substack{(i,j):\ztrue(i)\neq \ztrue(j)\\ \tz(i)= \tz(j)}}\exp\left(-\frac{1}{4}\theta_i\theta_j(\sqrt{p}-\sqrt{q})^2\right).
\end{align*}

In what follows, 
we focus on upper bounding $B_1$ and the same upper bound holds for $B_2$ by essentially repeating the arguments.
For $B_1$, we have
\begin{align}
B_1&=\prod_{u=1}^k\prod_{u'=1}^k\prod_{\substack{\{i:\ztrue(i)=u,\\\tz(i)=u' \}}}\prod_{\substack{\{j:\ztrue(j)=u,\\\tz(j)\neq u'\}}}\exp\bigg(-\frac{1}{4}\theta_i\theta_j(\sqrt{p}-\sqrt{q})^2\bigg)\label{eqn:simplify1}\\
&= \prod_{u=1}^k\prod_{u'\neq u}\exp\bigg(-\frac{1}{4}\sum_{\substack{\{i:\ztrue(i)=u,\\\tz(i)=u' \}}}\theta_i\sum_{\substack{\{j:\ztrue(j)=u,\\\tz(j)\neq u'\}}}\theta_j(\sqrt{p}-\sqrt{q})^2\bigg)\nonumber\\
&\;\;\;\;\times\prod_{u=1}^k\prod_{u'=u}\exp\bigg(-\frac{1}{4}\sum_{\substack{\{i:\ztrue(i)=u,\\\tz(i)=u' \}}}\theta_i\sum_{\substack{\{j:\ztrue(j)=u,\\\tz(j)\neq u'\}}}\theta_j(\sqrt{p}-\sqrt{q})^2\bigg)\nonumber
\\
&\leq \prod_{u=1}^k\prod_{u'\neq u}\exp\bigg(-\frac{1}{4}\sum_{\substack{\{i:\ztrue(i)=u,\\\tz(i)=u' \}}}\theta_i\sum_{\substack{\{j:\ztrue(j)=u,\\\tz(j)= u\}}}\theta_j(\sqrt{p}-\sqrt{q})^2\bigg)\nonumber\\
&\;\;\;\;\times\prod_{u=1}^k\prod_{u'\neq u}\exp\bigg(-\frac{1}{4}\sum_{\substack{\{i:\ztrue(i)=u,\\\tz(i)=u \}}}\theta_i\sum_{\substack{\{j:\ztrue(j)=u,\\\tz(j)= u'\}}}\theta_j(\sqrt{p}-\sqrt{q})^2\bigg). \label{eqn:simplify2}
\end{align}

\begin{figure}[!tb]
\includegraphics[trim=0mm 20mm 20mm 0mm, clip, width=0.35\textwidth]{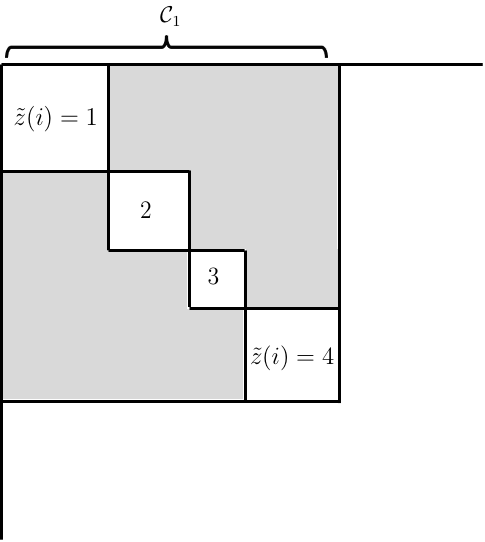}
\quad\quad\quad\quad\quad
\includegraphics[trim=0mm 20mm 20mm 0mm, clip, width=0.35\textwidth]{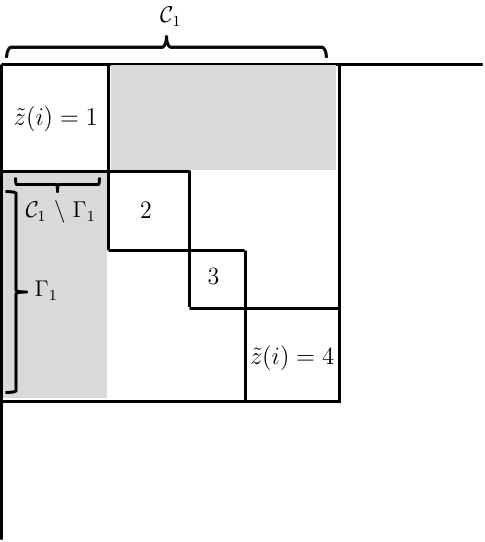}
\caption{Illustration of reduction to (\ref{eqn:simplify2}) when $k=4$. 
Only nodes from the first community are shown,
which are rearranged according to $\tilde{z}$.  
In the left panel the gray regions correspond to terms in (\ref{eqn:simplify1}). 
In the right panel the gray regions correspond to terms in (\ref{eqn:simplify2}). 
Note that the area of the gray regions in the left panel is larger than the one in the right panel. \label{fig:illustration}}
\end{figure}

\prettyref{fig:illustration} illustrates why the inequality in (\ref{eqn:simplify2}) holds. 
Furthermore, we notice that
\begin{align}
(\ref{eqn:simplify2})&= \prod_{u=1}^k\prod_{u'\neq u}\exp\bigg(-\frac{1}{2}\sum_{\substack{\{i:\ztrue(i)=u,\\\tz(i)=u' \}}}\theta_i\sum_{\substack{\{j:\ztrue(j)=u,\\\tz(j)= u\}}}\theta_j(\sqrt{p}-\sqrt{q})^2\bigg)\nonumber\\
&=\prod_{u=1}^k\exp\bigg(-\frac{1}{2}\sum_{i\in \Gamma_u}\theta_i\sum_{j\in \mathcal{C}_u\setminus \Gamma_u}\theta_j(\sqrt{p}-\sqrt{q})^2\bigg)\nonumber\\
&=\prod_{u=1}^k\exp\bigg(-\frac{1}{2}\sum_{i\in \Gamma_u}\theta_i\bigg(\sum_{i\in \mathcal{C}_u}\theta_i-\sum_{i\in \Gamma_u}\theta_i\bigg)(\sqrt{p}-\sqrt{q})^2\bigg).\label{eq:simplify3}
\end{align}
To further bound the right side of \eqref{eq:simplify3}, recall that $\theta_{\min} = \min_i \theta_i = \Omega(1)$.
Then 
\begin{align*}
\sum_{i\in \mathcal{C}_u}\theta_i-\sum_{i\in \Gamma_{u}}\theta_i = \sum_{\mathcal{C}_u\backslash \Gamma_u} \theta_i \geq  (|\mathcal{C}_u|-|\Gamma_u|)\theta_{\min}\geq  
\frac{|\mathcal{C}_u|\theta_{\min}}{2 }.
\end{align*}
{Here the last inequality holds since $|\Gamma_u|\leq |\Gamma| \leq m' = o(\nmin) \leq \frac{1}{2}|\mathcal{C}_u|$.}
Together with the property of the function $x(1-x), x\in[0,1]$, 
when $\folds \geq \frac{5}{\theta_{\min}}$,
we have
\begin{equation}
	\label{eq:mle-case2}
\begin{aligned}
\sum_{i\in \Gamma_{u}}\theta_i\bigg(\sum_{i\in \mathcal{C}_u}\theta_i-\sum_{i\in \Gamma_{u}}\theta_i\bigg)
& \geq \sum_{i\in \tau_u(\Gamma_{u})}\theta_i \bigg(\sum_{i\in \mathcal{C}_u}\theta_i-\sum_{i\in \tau_u(\Gamma_{u})}\theta_i\bigg)\\
& \geq \sum_{i\in \tau_u(\Gamma_{u})}\theta_i \bigg(\sum_{i\in \mathcal{C}_u}\theta_i-2\eta\nmin\bigg).	
\end{aligned}
\end{equation}
Here, the first inequality holds since $\sum_{i\in \tau_u(\Gamma_u)}	\theta_i \leq |\tau_u(\Gamma_u)| \max_{i\in \mathcal{C}_u^-}\theta_i \leq 2(M^{-1}|\mathcal{C}_u|+1)\leq \frac{1}{2}|\mathcal{C}_u|\theta_{\min}$.
The second inequality is due to \eqref{eq:Cu-}, $\nmin\geq \frac{n}{\beta k}-1$ and the fact $|\tau_u(\Gamma_u)|\leq |\Gamma_u| \leq \eta n/k$ in the current case.
Thus
\begin{align*}
B_1&\leq \prod_{u=1}^k\exp\left(-\frac{1}{2}\sum_{i\in \Gamma_{u}}\theta_i\left(\sum_{i\in \mathcal{C}_u}\theta_i-\sum_{i\in \Gamma_{u}}\theta_i\right)(\sqrt{p}-\sqrt{q})^2\right)\\
&\leq \prod_{u=1}^k\exp\left(-\frac{1}{2}
\sum_{i\in \tau_u(\Gamma_{u})}\theta_i\left(\sum_{i\in \mathcal{C}_u}\theta_i-2\eta\nmin\right)
(\sqrt{p}-\sqrt{q})^2\right)\\
&\leq \prod_{i\in\tau(\Gamma)}\exp\left(-\frac{1}{2}\theta_i\left(\sum_{j\in \mathcal{C}_{z(i)}}\theta_j-2\eta\nmin\right)(\sqrt{p}-\sqrt{q})^2\right)\\
&\leq \prod_{i\in\tau(\Gamma)}\exp\left(-\frac{1}{2}\theta_i(1-\delta-2\eta)
\nmin(\sqrt{p}-\sqrt{q})^2\right).
\end{align*}
Here, the last inequality is due to the approximate normalization constraint on the $\theta_i$'s.
Thus with the same bound on $B_2$ we obtain that for any $\tz$ such that $M < n\ell(\tz, \ztrue) \leq m'$,
\begin{align*}
\mathbb{P}\left(L(\tz)>L(\ztrue)\right)&\leq  \prod_{i\in\tau(\Gamma)}\exp\left(-\theta_i\left(1-\delta-2\eta \right)\nmin(\sqrt{p}-\sqrt{q})^2\right).
\end{align*}
Since for any $i\in \mathcal{C}_u^-$, $|\tau_u^{-1}(i)|\leq \folds$, we obtain that $|\tau(\Gamma)|\geq m/\folds$, and so we have 
\begin{align*}
& \mathbb{P}(n\ell(\hat{z},\ztrue)=m) \\
&\leq \sum_{\tilde{z}:|\Gamma|=m}k^m\prod_{i\in \tau(\Gamma)}\exp\left(-\theta_{i}(1-\delta-2\eta)\nmin(\sqrt{p}-\sqrt{q})^2\right) \\
&\leq {m\folds \choose m}{m\choose m/\folds}k^m\frac{1}{(m/M)!}\left(\sum_{i=1}^n
\exp\left(-\theta_{i}(1-\delta-2\eta)\nmin(\sqrt{p}-\sqrt{q})^2\right)\right)^{m/\folds} \\
&\leq (ek\folds)^{m}\left(\frac{e^2MnR_{\delta+2\eta}}{m/\folds}\right)^{m/\folds}.
\end{align*}
Here, the second inequality is based on counting and the details are as follows. Note that each term in $\prod_{i\in\tau(\Gamma)}$ is a product of at least $m/\folds$ terms.  First, there are at most ${m\choose m/\folds}$ of sets $\tau(\Gamma)$ that map to the same $m/\folds$-product. Then, there are at most ${m\folds\choose m}$ of sets $\Gamma$ that map to the same $\tau(\Gamma)$ (recall that for any $i\in \mathcal{C}_u^-$, $|\tau_u^{-1}(i)|<\folds$). For each $m/\folds$-product, it appear at most $(m/\folds)!$ times from the expansion of $nR_{\delta+2\eta}$, and that explains the existence of $1/(m/\folds)!$.
The last inequality holds since $\binom{n}{m}\leq \big(\frac{en}{m}\big)^m$ and $n!\geq \sqrt{2\pi}n^{n+1/2}e^{-n}$. 

\smallskip
\subparagraph{Case 3: $m> m'$} 
{In this case, we cannot use the same argument as in case 2 since \eqref{eq:mle-case2} does not necessarily hold.}
To proceed,
let $\Gamma_{u,u'}=\{i\in[n]:\ztrue(i)=u,\tz(i)=u'\}$ for any $u,u'\in[n]$. 
We have
\begin{align*}
B_1&=\prod_{u=1}^k\prod_{u'=1}^k\exp\Bigg(-\frac{1}{4}\sum_{\substack{\{i:\ztrue(i)=u,\\\tz(i)=u' \}}}\theta_i\sum_{\substack{\{j:\ztrue(j)=u,\\\tz(j)\neq u'\}}}\theta_j(\sqrt{p}-\sqrt{q})^2\Bigg)\\
&=\prod_{u=1}^k\prod_{u'=1}^k\exp\Bigg(-\frac{1}{4}\sum_{i\in \Gamma_{u,u'}}\theta_i\sum_{j\in \mathcal{C}_u\setminus \Gamma_{u,u'}}\theta_j(\sqrt{p}-\sqrt{q})^2\Bigg)\\
&=\prod_{u=1}^k\prod_{u'=1}^k\exp\Bigg(-\frac{1}{4}\sum_{i\in \Gamma_{u,u'}}\theta_i\bigg(\sum_{j\in \mathcal{C}_u}\theta_j-\sum_{j\in \Gamma_{u,u'}}\theta_j\bigg)(\sqrt{p}-\sqrt{q})^2\Bigg).
\end{align*}



To further proceed, we need to lower bound all $|\mathcal{C}_u\setminus \Gamma_{u,u'}|$ for all $u\neq u'$. 
To this end, we essentially follow the arguments leading to Lemma A.1 of \cite{zhang2015minimax}.
Let $n_{\max}$ and $n_{\min}$ be the maximum and the minimum community sizes.
We argue that we must have 
$|\mathcal{C}_u\setminus \Gamma_{u,u'}|\geq n_{\min}/9$ for all $u'\neq u$.
Indeed, if this were not the case, we could switch the labels $u$ and $u'$ in $z$.
This could reduce the Hamming distance between $z$ and $\tz$ by at least (see \prettyref{fig:Hamming} for illustration) 
\begin{align*}
& \hskip -2em |\Gamma_{u,u'}| - |\mathcal{C}_u \setminus \Gamma_{u,u'}| - |\sth{i:\tz(i) =u'}\setminus \Gamma_{u,u'}|\\
& \geq n_{\min}-\frac{1}{9}n_{\min} - \frac{1}{9}n_{\min} - (n_{\max} - (n_{\min}-\frac{1}{9}n_{\min}))\\
& \geq \frac{n}{k}\Big( \frac{5}{3\beta} - \beta \Big) > 0.
\end{align*}
Here, the last inequality holds when $1\leq \beta < \sqrt{5/3}$.
This leads to a contradiction since by definition, no permutation of the labels should be able to reduce $\ell(\tz,z)$.  

\begin{figure}[!tb]
\centering
\includegraphics[width=0.45\textwidth]{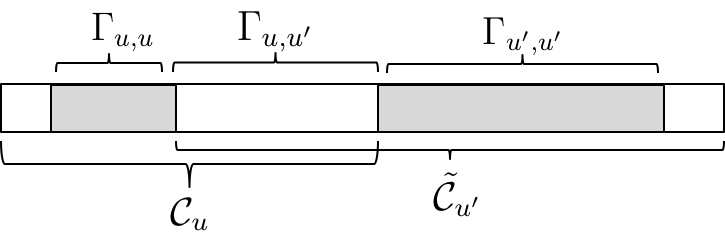}\quad\quad
\includegraphics[width=0.45\textwidth]{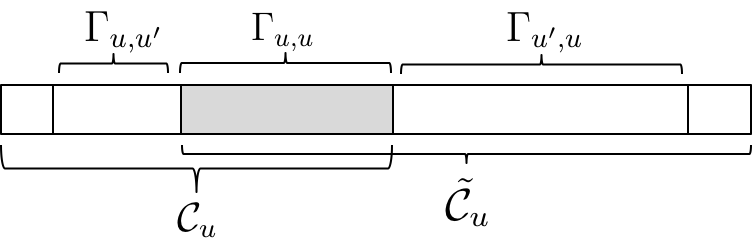}
\caption{In the left panel we display nodes in $\mathcal{C}_u\cup\tilde{\mathcal{C}}_{u'}$, where $\tilde{\mathcal{C}}_{u'}=\{i:\tz(i)=u'\}$. We also define  $\tilde{\mathcal{C}}_{u}$ in the same way. The gray parts indicates nodes correctly clustered, i.e., $\{i\in\mathcal{C}_u\cup\tilde{\mathcal{C}}_{u'}: \tz(i)=\ztrue(i)\}$. Note $|\Gamma_{u,u}|\leq |\mathcal{C}_u\setminus \Gamma_{u,u'}|$ and $|\Gamma_{u',u'}|\leq |\tilde{\mathcal{C}}_{u'}\setminus \Gamma_{u',u'}|$. The right panel displays the same nodes but after the labels $u$ and $u'$ flipped, and the gray part indicating nodes correctly clustered after flipping. \label{fig:Hamming}}
\end{figure}

%

In each $\Gamma_{u,u'}, \forall 1\leq u,u'\leq k$ we can find an arbitrary subset $\Gamma'_{u,u'}\subset \Gamma_{u,u'}$ such that $|\Gamma'_{u,u'}|=\eta |\Gamma_{u,u'}|$. In this way $\Gamma'\triangleq \cup_{u\in [k]}\cup_{u'\neq u}\Gamma'_{u,u'}$ satisfies $|\Gamma'|= \eta |\Gamma|\leq \eta m$ and $\sum_{i\in \Gamma'_{u,u'}}\theta_i\leq 2|\Gamma'_{u,u'}|\leq 2\eta |\mathcal{C}_u|$.


%

Note that for $\eta = o(1)$,
\begin{align*}
\sum_{i\in \mathcal{C}_u}\theta_i-\sum_{i\in \Gamma_{u,u'}}\theta_i 
\geq \theta_{\min}|\mathcal{C}_u\setminus \Gamma_{u,u'}|\geq \frac{\nmin\theta_{\min}}{9}
\geq 2\eta \frac{\beta n}{k} \geq 2\eta |\mathcal{C}_u|.
\end{align*}
Together with the property of the function $x(1-x), x\in[0,1]$, we have
\begin{align*}
& \sum_{i\in \Gamma_{u,u'}}\theta_i\bigg(\sum_{i\in \mathcal{C}_u}\theta_i-\sum_{i\in \Gamma_{u,u'}}\theta_i\bigg)\geq \sum_{i\in \Gamma'_{u,u'}}\theta_i \bigg(\sum_{i\in \mathcal{C}_u}\theta_i-\sum_{i\in \Gamma'_{u,u'}}\theta_i\bigg)\\
& ~~~~~~~ \geq \sum_{i\in \Gamma'_{u,u'}}\theta_i \big((1-\delta)|\mathcal{C}_u|-2\eta|\mathcal{C}_u|\big)
\geq \sum_{i\in \Gamma'_{u,u'}}\theta_i  (1-\delta-2\eta)\nmin .
\end{align*}
Then
\begin{align*}
B_1&\leq \prod_{u=1}^k\prod_{u'=1}^k\exp\bigg(-\frac{1}{4}\sum_{i\in \Gamma'_{u,u'}}\theta_i  (1-\delta-2\eta)\nmin(\sqrt{p}-\sqrt{q})^2\bigg) \\
&= \prod_{u=1}^k\exp\bigg(-\frac{1}{4}\sum_{i\in \Gamma'_{u}}\theta_i  (1-\delta-2\eta)\nmin(\sqrt{p}-\sqrt{q})^2\bigg) \\
&\leq \prod_{i\in \Gamma'}\exp\left(-\frac{1}{4}\theta_i(1-\delta-2\eta)\nmin(\sqrt{p}-\sqrt{q})^2\right)
\end{align*}
Thus with the same bound on $B_2$ we get
\begin{align*}
\mathbb{P}\left(L(\tz)>L(\ztrue)\right)&\leq  \prod_{i\in \Gamma'}\exp\left(-\frac{1}{2}\theta_i\left(1-\delta-2\eta\right)\nmin(\sqrt{p}-\sqrt{q})^2\right).
\end{align*}
Note we have an extra $1/2$ factor inside the exponent compared with Case 2. Since for each $\Gamma$ we can find a subset $\Gamma'$ with $|\Gamma'|=\eta|\tau(\Gamma)|\geq \eta m/\folds$ satisfying the above inequality, we have
\begin{align*}
& \p(n\ell(\hat{z},\ztrue)=m)\\
&\leq \sum_{\tilde{z}:|\Gamma|=m}k^m\prod_{i\in \Gamma'}\exp\left(-\frac{1}{2}\theta_{i}(1-\delta-2\eta)\nmin(\sqrt{p}-\sqrt{q})^2\right)\\
&\leq k^m \binom{m\folds}{m}\binom{m}{\eta m/\folds}\frac{1}{(\eta m/\folds)!}\Big(\sum_{i=1}^n\exp\Big(-\frac{1}{2}\theta_{i}(1-\delta-2\eta)\nmin(\sqrt{p}-\sqrt{q})^2\Big)\Big)^{\eta m/\folds}\\
&\leq (ek\folds)^m\left(\frac{e^2\folds nR^{1/2}_{\delta+2\eta}}{\eta^2m/\folds}\right)^{\eta m/\folds},
\end{align*}
where the last equality is due to Cauchy-Schwarz.

To sum up, till now we have derived the probability $\p(n\ell(\hat{z},\ztrue)=m)$ for each $1\leq m\leq n$ as follows
\begin{equation}
\p(n\ell(\hat{z},\ztrue)=m)\leq
\begin{cases}
(knMR_{\delta+2\eta})^m, & 1\leq m\leq \folds\\
(ek\folds)^{m}\left(\frac{e^2MnR_{\delta+2\eta}}{m/\folds}\right)^{m/\folds}, 
& \folds<  m\leq \frac{\eta n}{k}\\
(ek\folds)^m\left(\frac{e^2MnR^{1/2}_{\delta+2\eta}}{\eta^2m/M}\right)^{\eta m/M},
& m>\frac{\eta n}{k}.
\end{cases}
\label{eq:ind-prob-bd}
\end{equation}

\paragraph{Step 3: bounding $\mathbb{E}\ell(\hat{z},\ztrue)$}
As we have pointed out in the proof outline, we shall combine \eqref{eq:basic-mle} with \eqref{eq:ind-prob-bd} in this step to finish the proof. 
To this end, 
we divide the argument into three cases according to different possible growth rates of $R_{\delta+2\eta}$.

\smallskip
\subparagraph{Case 1: $R_{\delta+2\eta}\leq \frac{1}{2(ek\folds)^{\folds+2} n}$} Recall that $m'=\eta n/k$. Then 
\begin{align*}
\mathbb{E}n\ell(\hat{z},\ztrue)=\mathbb{P}(n\ell(\hat{z},\ztrue)=1)+\sum_{m=2}^{m'} m\mathbb{P}(n\ell(\hat{z},\ztrue)=m)+\sum_{m=m'+1}^nm\mathbb{P}(n\ell(\hat{z},\ztrue)=m).
\end{align*}
We have $\mathbb{P}(n\ell(\hat{z},\ztrue)=1)=kMnR_{\delta+2\eta}$ which is upper bounded by $1/2$. Together with $(ek\folds)^\folds e^2\folds nR_{\delta+2\eta}\leq 1/2$,  we have
\begin{align*}
\sum_{m=2}^{m'} m\mathbb{P}(n\ell(\hat{z},\ztrue)=m)&=\sum_{m=2}^{\folds} m\mathbb{P}(n\ell(\hat{z},\ztrue)=m)+\sum_{m=\folds +1}^{m'} m\mathbb{P}(n\ell(\hat{z},\ztrue)=m)\\
&\leq \sum_{m=2}^{\folds} m 2^{-m} + \sum_{m=\folds +1}^{m'}(ek\folds)^\folds (e^2\folds nR_{\delta+2\eta})m2^{-(m-\folds)/\folds}\\
&\leq C_1(ek\folds)^\folds e^2\folds nR_{\delta+2\eta},
\end{align*}
for some constant $C_1>1$ where the last inequality is due to the properties of power series. For $m > m'$ we have
\begin{align*}
\frac{\mathbb{P}(n\ell(\hat{z},\ztrue)=m)}{nR_{\delta+2\eta}}&\leq (ek\folds)^m\left(\frac{e^2\folds nR^{1/2}_{\delta+2\eta}}{\eta^2m/\folds}\right)^{\eta m/\folds-2}\\
&\leq (ek\folds)^m\left(\left(\frac{e^2\folds n}{\eta^2 m/\folds}\right)^\frac{1}{2}\left(\frac{e^2\folds nR_{\delta+2\eta}}{\eta^2 m/\folds}\right)^\frac{1}{2}\right)^{\eta m/\folds-2}.
\end{align*}
We are to show that the above ratio is upper bounded by $e^{-m}$. This is because $e^2\folds n/(\eta^2 m/\folds)\leq \eta^{-3}e^2M^2k$ since $m\geq \eta n/k$ and $e^2\folds nR_{\delta+2\eta}/(\eta^2 m/\folds)\leq 1/(2(k\folds)^\folds \eta^3n)$ since $nR_{\delta+2\eta}\leq 1/(2(e\folds)^\folds)$. Then for some constant $C_2>0$ we have
\begin{align*}
\frac{\mathbb{P}(n\ell(\hat{z},\ztrue)=m)}{nR_{\delta+2\eta}}\leq (ek\folds)^m\left(\frac{C_2}{\eta^3	n}\right)^{\frac{\eta m}{2\folds}}=\exp\left(m\log(ek\folds)+\frac{\eta m}{2\folds}\log\left(\frac{C_2}{\eta^3 n}\right)\right)\leq e^{-m},
\end{align*}
where the last inequality is due to the fact that $k\leq n^\eta$. By the property of power series we have
\begin{align*}
\sum_{m=m'+1}^nm\mathbb{P}(n\ell(\hat{z},\ztrue)=m)\leq nR_{\delta+2\eta}\sum_{m=m'+1}^n me^{-m}\leq C_3nR_{\delta+2\eta},
\end{align*}
for some constant $C_3>0$. Finally by Jensen's inequality and the assumption $\log k=o(I)$,
\begin{align*}
\mathbb{E}n\ell(\hat{z},\ztrue)\leq kMnR_{\delta+2\eta}+C_1(ek\folds)^\folds e^2\folds nR_{\delta+2\eta}+C_3nR_{\delta+2\eta} = n\exp(-(1-o(1))I).
\end{align*}

\subparagraph{Case 2: $R_{\delta+2\eta}\geq \frac{\folds\log n}{(ek\folds)^{\folds+2}n}$} 
Let $m_0=2(ek\folds)^{\folds+2}nR_{\delta+2\eta}$. 
Recall that $I\to\infty$ and that $\log k=o(I)$. 
So $\eta^{-1}=o(I)$ and $m_0\leq m'$.
We have
$$
\mathbb{E}\ell(\hat{z},\ztrue)\leq \frac{m_0}{n} +\sum_{m=m_0+1}^{m'}\mathbb{P}\left(n\ell(\hat{z},\ztrue)=m\right)+\sum_{m>m'}\mathbb{P}\left(n\ell(\hat{z},\ztrue)=m\right).
$$
To obtain the last display, we divide both sides of \eqref{eq:basic-mle} by $n$, replace all the $m$'s in front of the probabilities in the summation by $n$ and then upper bound the first $m_0$ probabilities by one.
To further bound the right side of the last display, we have
\begin{align*}
& \sum_{m=m_0+1}^{m'}\mathbb{P}\left(n\ell(\hat{z},\ztrue)=m\right) \leq \sum_{m=m_0+1}^{m'}\left(\frac{(ek\folds)^{\folds+2} nR_{\delta+2\eta}}{m_0}\right)^{m/\folds}\\
& \leq \sum_{m=m_0+1}^{m'}2^{-m/\folds}\leq M2^{-m_0/\folds}.
\end{align*}
Since $m_0\geq 2\folds\log n$, we have $2^{-m/\folds}\leq 2^{-2\log n}\leq m_0/n$. Thus
\begin{align*}
\sum_{m=m_0+1}^{m'}\mathbb{P}\left(n\ell(\hat{z},\ztrue)=m\right)\leq \frac{Mm_0}{n}.
\end{align*}
For $m\geq m'$, we are going to show $\mathbb{P}\left(n\ell(\hat{z},\ztrue)=m\right)\leq 2^{-\eta m/\folds}$. We have
\begin{align*}
\mathbb{P}\left(n\ell(\hat{z},\ztrue)=m\right)\leq \left(\frac{(ek\folds)^{\frac{\folds}{\eta}}e^2\folds^2 nR^{1/2}_{\delta+2\eta}}{\eta^2m'}\right)^{\eta m/\folds}\leq \left(\frac{(ek\folds)^{\frac{\folds}{\eta}+3}R^{1/2}_{\delta+2\eta}}{\eta^3}\right)^{\eta m/\folds}.
\end{align*}
Since $R_{\delta+2\eta}\leq \exp(-(1-\delta-2\eta)I)$ by Jensen's inequality, $\log k=o(I)$ and $\eta^{-1} =o(I)$, we have $\eta^{-3}(ek\folds)^{\frac{\folds}{\eta}+3}R^{1/2}_{\delta+2\eta}\leq 1/2$. Then
\begin{align*}
\sum_{m>m'}\mathbb{P}\left(n\ell(\hat{z},\ztrue)=m\right)\leq \sum_{m>m'}2^{-\eta m/\folds}\leq \eta^{-1} M2^{-\eta^2n/\folds}
\leq \frac{m_0}{n}.
\end{align*}
Thus by Jensen's inequality $\mathbb{E}\ell(\hat{z},\ztrue)\leq (2+\folds)m_0/n\leq \exp(-(1-o(1))I)$.

\smallskip
\subparagraph{Case 3: $\frac{1}{2(ek\folds)^{\folds+2} n}<R_{\delta+2\eta}< \frac{\folds\log n}{(ek\folds)^{\folds+2}n}$}
Let $m_0=2\folds \log n$.
As we have shown in Case 2, $M\log n\leq m'$. Then 
\begin{align*}
\mathbb{E}\ell(\hat{z},\ztrue)\leq \frac{m_0}{n}+\sum_{m=m_0+1}^{m'}\mathbb{P}\left(n\ell(\hat{z},\ztrue)=m\right)+\sum_{m>m'}\mathbb{P}\left(n\ell(\hat{z},\ztrue)=m\right).
\end{align*}
We have
\begin{align*}
 \sum_{m=m_0+1}^{m'}\mathbb{P}\left(n\ell(\hat{z},\ztrue)=m\right)&\leq \sum_{m=m_0+1}^{m'}\left(\frac{(ek\folds)^{\folds+2} nR_{\delta+2\eta}}{m_0}\right)^{m/\folds}\\
 & \leq  \sum_{m=m_0+1}^{m'}2^{-m/\folds}\leq M2^{-m_0}\leq \frac{m_0}{n}.
\end{align*}
For $m>m'$, in Case 2 we have shown $\sum_{m>m'}\mathbb{P}\left(n\ell(\hat{z},\ztrue)=m\right)\leq \sum_{m>m'}2^{-\eta m/\folds}\leq \eta^{-1} M2^{-\eta^2n/\folds}$, which is also upper bounded by $m_0/n$. Together we have $\mathbb{E}\ell(\hat{z},\ztrue)\leq 3m_0/n$. Since $2(ek\folds)^{\folds+2}R_{\delta+2\eta}\geq 1/n $ and $\log k=o(I)$, we have $n\exp(-(1-o(1))I)\geq M\log n$ for some positive sequence $o(1)$. Then $\mathbb{E}\ell(\hat{z},\ztrue)\leq \exp(-(1-o(1))I)$.




\appendix

\section{Additional Proofs of Main Results}

\subsection{Proof of \prettyref{thm:upper-nbar} for $k=2$}
\label{sec:minimax-2}

By the definition of the loss function,  $n\ell(\tz,\ztrue)\leq n/2$ for any $\tilde{z}\in[2]^n$. Therefore, we only need to calculate $\mathbb{P}(n\ell(\hat{z},z)=m)$ for $1\leq m \leq n/2$. 
We will keep using the definitions $\Gamma_{u,v}=\{i:\ztrue(i)=u,\tz(i)=v\}$, $\mathcal{C}_u=\Gamma_{u,1}\cup \Gamma_{u,2}$ and $\Gamma=\Gamma_{1,2}\cup \Gamma_{2,1}$ for all $u,v\in[2]$.
Recall in \prettyref{eq:mle-1} we have shown
\begin{align*}
\p\left(L(\tz)>L(\ztrue)\right) \leq \prod_{\substack{i< j\\ \tilde{Y}_{ij}\neq Y_{ij}}}\exp\left(-\frac{1}{2}\theta_i\theta_j(\sqrt{p}-\sqrt{q})^2\right).
\end{align*}
Since
\begin{eqnarray*}
\sum_{\substack{i< j\\ \tilde{Y}_{ij}\neq Y_{ij}}}\theta_i\theta_j&=&\sum_{i\in \Gamma_{1,1}}\theta_i\sum_{i\in \Gamma_{1,2}}\theta_i+\sum_{i\in \Gamma_{1,1}}\theta_i\sum_{i\in \Gamma_{2,1}}\theta_i+\sum_{i\in \Gamma_{1,2}}\theta_i\sum_{i\in \Gamma_{2,1}}\theta_i+\sum_{i\in \Gamma_{2,1}}\theta_i\sum_{i\in \Gamma_{2,2}}\theta_i\\
&\geq& \sum_{i\in\Gamma}\theta_i\left((1-\delta)n-\sum_{i\in\Gamma}\theta_i\right),
\end{eqnarray*}
we have
\begin{equation}
\p\left(L(\tz)>L(\ztrue)\right) \leq \exp\left(-\frac{1}{2}\sum_{i\in\Gamma}\theta_i\left((1-\delta)n-\sum_{i\in\Gamma}\theta_i\right)(\sqrt{p}-\sqrt{q})^2\right).\label{eq:nimanima}
\end{equation}
Denote $m'=\eta n$ for some $\eta=o(1)$ satisfying $\eta^{-1}=o(I)$. We define $\tau$ exactly the same way as in \prettyref{sec:messi}. We use the notation
\begin{equation*}
R_{t}=\frac{1}{n}\sum_{i=1}^n\exp\left(-(1-t)\theta_i\frac{n}{2}(\sqrt{p}-\sqrt{q})^2\right).	
\end{equation*}
Recall the constant $M$ used in \prettyref{sec:messi}.
\newline
\textbf{Case 1: $1\leq m\leq \folds$}. By (\ref{eq:nimanima}), we have
\begin{align*}
\mathbb{P}(n\ell(\hat{z},\ztrue)=m) \leq \sum_{|\Gamma|=m}\exp\left(-\frac{1}{2}\sum_{i\in\Gamma}\theta_i\left((1-\delta)n-\sum_{i\in\Gamma}\theta_i\right)(\sqrt{p}-\sqrt{q})^2\right).
\end{align*}
Using the argument in \prettyref{sec:messi}, we have $\p(n\ell(\tz,\ztrue)=m)\leq (2n\folds R_{\delta+2\eta})^m$.
\newline
\textbf{Case 2: $M\leq m\leq m'$}.
We have $\sum_{i\in\tau(\Gamma)}\theta_i\leq 2|\tau(\Gamma)|\leq 2\eta n$ due to \prettyref{eq:Cu-}. Note that $n-\sum_{i\in\Gamma}\theta_i=\sum_{i\in \Gamma^c}\theta_i\geq |\Gamma^c|\theta_{\min}\geq n\theta_{\min}/2$. For any $m\leq m'$, using the monotone property of $x(1-x)$ for $x\in[0,1]$, we have 
\begin{align*}
\sum_{i\in\Gamma}\theta_i\left((1-\delta)n-\sum_{i\in\Gamma}\theta_i\right)\geq \sum_{i\in\tau(\Gamma)}\theta_i\left((1-\delta)n-\sum_{i\in\tau(\Gamma)}\theta_i\right)\geq \sum_{i\in\tau(\Gamma)}\theta_i(1-\delta-2\eta)n.
\end{align*}
Thus, by (\ref{eq:nimanima}), we have
\begin{align*}
\p\left(L(\tz)>L(\ztrue)\right) &\leq \prod_{i\in\tau(\Gamma)}\exp\left(-\theta_i(1-\delta - 2\eta)\frac{n}{2}(\sqrt{p}-\sqrt{q})^2\right).
\end{align*}
Using the argument in \prettyref{sec:messi}, we have
\begin{align*}
\p(n\ell(\tz,\ztrue)=m)\leq (2e\folds)^{m}\left(\frac{e^2MnR_{\delta+2\eta}}{m/\folds}\right)^{m/\folds}.
\end{align*}
\newline
\textbf{Case 3: $m> m'$}. Under this scenario, we can take an arbitrary subset $\Gamma'\subset  \tau(\Gamma)$ such that $|\Gamma'|=\eta m/\folds$, which leads to $\sum_{i\in\Gamma'}\theta_i\leq 2\eta n$. Recall $n-\sum_{i\in\Gamma}\theta_i=\sum_{i\in \Gamma^c}\theta_i\geq |\Gamma^c|\theta_{\min}\geq n\theta_{\min}/2$. Using (\ref{eq:nimanima}), together with the property of $x(1-x)$ for $x\in[0,1]$, we have
\begin{align*}
\p\left(L(\tz)>L(\ztrue)\right) &\leq \exp\left(-\frac{1}{2}\sum_{i\in\Gamma'}\theta_i\left((1-\delta)n-\sum_{i\in\Gamma'}\theta_i\right)(\sqrt{p}-\sqrt{q})^2\right)\\
&\leq \prod_{i\in\Gamma'}\exp\left(-\theta_i(1-\delta-2\eta)\frac{n}{2}(\sqrt{p}-\sqrt{q})^2\right).
\end{align*}
By the argument used in \prettyref{sec:messi}, we have
\begin{align*}
\p(n\ell(\tz,\ztrue)=m)\leq  (2e\folds)^m\left(\frac{e^2\folds nR_{\delta+2\eta}}{\eta^2m/\folds}\right)^{\eta m/\folds}.
\end{align*}
Note that the above rate involves $R_{\delta+2\eta}$ instead of $R_{\delta+2\eta}^{1/2}$ for the case $k\geq 3$ in \prettyref{sec:messi}. This results in a tighter bound for  $\p\left(L(\tz)>L(\ztrue)\right)$.

Finally by applying the same techniques used in \prettyref{sec:messi}, we obtain the desired bound for $\E \ell(\tz,\ztrue)$. 


\subsection{Proof of Theorem \ref{thm:lower-nbar}}


We only state the proof for the case $k\geq 3$. The proof for the case $k=2$ can be derived using essentially the same argument.
For a label vector, recall the notation $n_u(z)=|\{i\in[n]:z(i)=u\}|$.
Under Condition N, there exists a $z^*\in [k]^n$ such that $n_1(z^*) \leq n_2(z^*) \leq n_3(z^*)\leq \cdots\leq n_k(z^*)$ with $n_1(z^*) = n_2(z^*)=\floor{n/(\beta k)}$, 
and that $(n_u(z^*))^{-1}{\sum_{i:z^*(i)=u}\theta_i} \in (1-\frac{\delta}{4}, 1+\frac{\delta}{4})$ for all $u\in [k]$. 

\smallskip
$1^\circ$
{As a first step, we define a community detection problem on a subset of the parameter space such that we can avoid the complication of label permutation.} 
To this end, 
given $z^*$, for each $u\in[k]$, let $T_u\subset\{i:z^*(i)=u\}$ with cardinality $\ceil{n_u(z^*)-\frac{\delta n}{4k^2\beta}}$ collect the indices of the largest $\theta_i$'s in $\{\theta_i:{z(i)=u}\}$. 
Let $T=\cup_{u=1}^k T_u$. 
Define
\begin{equation*}
\begin{aligned}
Z^*=\Big\{z\in [k]^n: &~~z(i)=z^*(i)\text{ for all }i\in T, \frac{n}{\beta k}-1 \leq n_u(z)\leq \frac{\beta n}{k}+1 \text{ for all }u\neq v\in [k]
\Big\}.		
\end{aligned}
\end{equation*}
Since $z^*\in Z^*$, the latter is not empty.
By the definition of $T$ and Condition N, $\max_{i\in T^c}\theta_i$ is bounded by a constant. Thus,
 for any $z$ such that $z(i)=z^*(i)$ for all $i\in T$, we have
\begin{equation*}
\frac{1}{n_u(z)}\sum_{\{i:z(i)=u\}}\theta_i\in (1-\delta, 1+\delta), \quad\text{for all }u\in [k].
\end{equation*}
Therefore, we can define a smaller parameter space $\calP_n^0 = \calP_n^0(\theta, p, q, k,\beta;\delta)\subset \calP_n(\theta, p, q, k,\beta;\delta)$ where
\begin{equation}
 \calP_n^0(\theta, p, q, k,\beta;\delta)
=\left\{P:P_{ij}=\theta_i\theta_jB_{z(i)z(j)}, z\in {Z}^*, B_{uu}=p,\forall u\in[k],\; B_{uv}=q,\forall u\neq v\right\}.
\end{equation}
So we have
\begin{equation}
	\label{eq:lowbd-reduce}
\inf_{\hat{z}}\sup_{\calP_n(\theta, p, q, k,\beta;\delta)}\mathbb{E} n\ell(\hat{z},z)
\geq 
\inf_{\hat{z}}\sup_{\calP_n^0}\mathbb{E} n\ell(\hat{z},z)
= \inf_{\hat{z}}\sup_{\calP_n^0}\mathbb{E}H(\hat{z},z),	
\end{equation}
where $H(\cdot,\cdot)$ is the Hamming distance. 
Here, the equality is due to the fact that for any two $z_1,z_2\in {Z}^*$ they share the same labels for all indices in $T$. Thus, we have $H(z_1,z_2)\leq \frac{1}{2}\delta \frac{n}{\beta k}$, and so when $\delta$ is small we have
$n\ell(z_1,z_2)=\inf_{\pi\in \Pi_k}H(\pi(z_1),z_2)= H(z_1,z_2)$.

\smallskip

$2^\circ$
We now turn to lower bounding the rightmost side of \eqref{eq:lowbd-reduce}, which relies crucially on our previous discussion in \prettyref{sec:testing}.
To this end, observe that
\begin{align}
\inf_{\hat{z}}\sup_{{\mathcal{P}}_0}\mathbb{E}H(\hat{z},z) &\geq \inf_{\hat{z}}\ave_{{Z^*}}\mathbb{E}H(\hat{z},z) 
\geq
\sum_{i\in T^c}\inf_{\hat{z}(i)}
\ave_{{Z^*}}\mathbb{P}(\hat{z}(i)\neq z(i)) \nonumber \\
&\geq c\frac{\delta}{k}n\frac{1}{|T^c|}\sum_{i\in T^c}\inf_{\hat{z}(i)}\ave_{{Z^*}}\mathbb{P}(\hat{z}(i)\neq z(i)),
\label{eq:lowbd-reduce-1}
\end{align}
for some constant $c>0$.
Here, $\ave$ stands for arithmetic average. 
The first inequality holds since minimax risk is lower bounded by Bayes risk.
The second inequality is due to the fact that for any $z\in Z^*$, $z(i)=z^*(i)$ for all $i\in T$, and so infimum can be taken over all $\hat{z}$ with $\hat{z}(i) = z^*(i)$ for $i\in T$.
The last inequality holds because $|T^c|\geq c\frac{ \delta n}{k}$ for some constant $c$ by its definition.

We now focus on lower bounding $\inf_{\hat{z}(i)}\ave_{{Z^*}}\mathbb{P}(\hat{z}(i)\neq z(i))$ for each $i\in T^c$.
Without loss of generality, suppose $1\in T^c$. 
Then we partition $Z^*$ into disjoint subsets $Z^* = \cup_{u = 1}^k Z^*_u$ where 
\begin{equation*}
Z^*_u=\{z\in Z^*: z(1)=u\},\quad u \in [k].
\end{equation*}
Note that for any $u\neq v$, there is a 1-to-1 correspondence between the elements in $Z^*_u$ and $Z^*_v$.
In particular, for each $z\in Z^*_u$, there exists a unique $z'\in Z^*_v$ such that $z(i) = z'(i)$ for all $i\neq 1$.
Thus, we can simultaneously index all $\{Z^*_u\}_{u=1}^k$ by the second to the last coordinates of the $z$ vectors contained in them.
We use $z_{-1}$ to indicate the subvector in $[k]^{n-1}$ excluding the first coordinate and collect all the different $z_{-1}$'s into a set $Z_{-1}$.
Then we have
\begin{align}
& \inf_{\hat{z}(1)}\ave_{{Z^*}}\mathbb{P}(\hat{z}(1)\neq z(1)) \nonumber \\
&~~ \geq \frac{1}{k(k-1)}\sum_{u< v}\inf_{\hat{z}(1)}\left(\ave_{Z^*_u}\mathbb{P}(\hat{z}(1)\neq u)+\ave_{Z^*_v}\mathbb{P}(\hat{z}(1)\neq v)\right) \nonumber \\
&~~ \geq \frac{1}{k(k-1)} 
\inf_{\hat{z}(1)}\left(\ave_{Z^*_1}\mathbb{P}(\hat{z}(1)\neq 1)+\ave_{Z^*_2}\mathbb{P}(\hat{z}(1)\neq 2)\right) \nonumber \\
&~~ \geq \frac{1}{k(k-1)} \frac{1}{|Z_{-1}|} \sum_{z_{-1}\in Z_{-1}}
\inf_{\hat{z}(1)}\left(\mathbb{P}_{z = (1, z_{-1})}(\hat{z}(1)\neq 1)+\mathbb{P}_{z=(2,z_{-1})}(\hat{z}(1)\neq 2)\right).
\label{eq:lowbd-reduce-2}
\end{align}

Note that by the definition of $z^*$ and $Z^*$, it is guaranteed that for either $(1, z_{-1})$ or $(2,z_{-1})$, $|\frac{n_1}{n_2}-1|=o(1)$.
Therefore, we can apply Lemma \ref{lem:t-lower-aprox} to bound from below each term in the summation of the rightmost side of the last display by 
$\exp\pth{-(1+\eta)\theta_1\frac{n}{\beta k}(\sqrt{p}-\sqrt{q})^2}$ for some $\eta=o(1)$. 
Together with \eqref{eq:lowbd-reduce} -- \eqref{eq:lowbd-reduce-2}, this implies that 
\begin{align*}
\inf_{\hat{z}}\sup_{\calP}\mathbb{E} \ell(\hat{z},z) 
& \geq c\frac{\delta}{k^3}\frac{1}{|T^c|} \sum_{i\in T^c} \exp\pth{-(1+\eta)\theta_i\frac{n}{\beta k}(\sqrt{p}-\sqrt{q})^2} \\
& \geq c\frac{\delta}{k^3}\frac{1}{n} \sum_{i=1}^n \exp\pth{-(1+\eta)\theta_i\frac{n}{\beta k}(\sqrt{p}-\sqrt{q})^2} \\
& \geq c\frac{\delta}{k^3} \exp\pth{-(1+\eta) I} 
= \exp\pth{-(1+o(1)) I}.
\end{align*}
Here, the first inequality is simple algebra. 
The second inequality holds since $T^c$ only contains within each community defined by $z^*$ the nodes with the smallest $\theta_i$'s.
The third inequality is a direct application of Jensen's inequality, and the last equality holds since $\log k=o(I)$ and $\log \frac{1}{\delta} = o(I)$.
This completes the proof.
\subsection{Proofs of \prettyref{lem:initial3} and \prettyref{cor:ini}}
We now prove Lemma \ref{lem:initial3} and Corollary \ref{cor:ini}, which characterize the performance of Algorithm \ref{alg:algl3}. To prove Lemma \ref{lem:initial3}, we need two auxiliary lemmas, whose proofs will be given in Appendix \ref{sec:pf-aux}.
In the rest of this part, we let 
$P = (P_{ij}) = (\theta_i\theta_j B_{z(i)z(j)})$ for notational convenience.

{The following lemma characterizes the connection between measure on misclassification and geometry of the point cloud. 
The result is not tied to any specific clustering algorithm or choice of norm.}

\begin{lemma}\label{lem:S-anderson}
Let $z\in[k]^n$ be the true label for a DCBM in $\calP_n'(\theta,p,q,k,\beta;\delta,\alpha)$.
Given any $\tilde{z}\in (\{0\}\cup[k])^n$, any $\{\tilde{v}_u\}_{u\in[k]}, \{V_i\}_{i\in[n]}\subset\mathbb{R}^n$ and any $b>0$, define
$$\tilde{V}_i=\tilde{v}_{\tilde{z}(i)}\quad\text{for all }i\in S_0^c,$$
where $S_0=\{i\in[n]:\tilde{z}(i)=0\}$.
Then, for any norm $\norm{\cdot}$ satisfying triangle inequality, as long as
\begin{equation}
\min_{z(i)\neq z(j)}\norm{V_i-V_j}\geq 2b,\label{eq:tttgap}
\end{equation}
we have
$$\min_{\pi\in\Pi_k}\sum_{\{i:\tilde{z}(i)\neq \pi(z(i))\}}\theta_i\leq \sum_{i\in S_0}\theta_i+(2\beta^2+1)\sum_{i\in S}\theta_i,$$
where $S=\left\{i\in S_0^c:\norm{\tilde{V}_i-V_i}\geq b\right\}$.
\end{lemma}

\begin{lemma}\label{lem:P-hat-bound}
Under the settings of Lemma \ref{lem:initial3}, let $\tau=C_1\left(np\norm{\theta}_{\infty}^2+1\right)$ for some sufficiently large $C_1>0$ in Algorithm \ref{alg:algl3}. Then, for any constant $C'>0$, there exists some $C>0$ only depending on $C_1,C'$ and $\alpha$ such that
$$\fnorm{\hat{P}-P}\leq C\sqrt{k(np\norm{\theta}_{\infty}^2+1)},$$
with probability at least $1-n^{-(1+C')}$ uniformly over $\calP_n'(\theta,p,q,k,\beta;\delta,\alpha)$.
\end{lemma}

\begin{proof}[Proof of Lemma \ref{lem:initial3}]
Let $P_i$ denote the $i\Th$ row of $P$ and $\bar{P}_i=\norm{P_i}_1^{-1}P_i$ the $\ell_1$ normalized row. 
By definition, for sufficiently large values of $n$,
\begin{equation}
\frac{pn}{2\beta k}\leq \frac{\norm{P_i}_1}{\theta_i}
=\sum_{j: j\neq i}\theta_jB_{z(i)z(j)}\leq 2\alpha np,\text{ for any }i\in[n],\label{eq:p1t}
\end{equation}
under the conditions $\delta=o(1)$ and $\norm{\theta}_{\infty}=o(n/k)$.

Note that $\bar{P}_i = \bar{P}_j$ when $z(i) = z(j)$.
Our first task is to lower bound $\norm{\bar{P}_i-\bar{P}_j}_1$ when $z(i)\neq z(j)$, which serves as the separation condition among different clusters. For any $i$ and $j$ such that $z(i)=u\neq v=z(j)$, we assume $\norm{P_i}_1/\theta_i\leq \norm{P_j}_1/\theta_j$ without loss of generality. Then, 
\begin{eqnarray}
\nonumber\norm{\bar{P}_i-\bar{P}_j}_1 &\geq& 
\sum_{l:z(l)=u} |{P}_{il} - {P}_{jl}| 
= \sum_{l:z(l)=u}\left|\frac{\theta_lB_{uu}}{\norm{P_i}_1/\theta_i}-\frac{\theta_lB_{uv}}{\norm{P_j}_1/\theta_j}\right| \\
\nonumber&=& \frac{1}{\norm{P_j}_1/\theta_j}\sum_{l:z(l)=u}\theta_l\left|\frac{\norm{P_j}_1/\theta_j}{\norm{P_i}_1/\theta_i}B_{uu}-B_{uv}\right| \\
\label{eq:magic}&\geq& \frac{p-q}{\norm{P_j}_1/\theta_j}\frac{n}{2\beta k} \\
\label{eq:up1t}&\geq& \frac{p-q}{4\alpha\beta kp}.
\end{eqnarray}
Here, (\ref{eq:magic}) holds since $\frac{\norm{P_j}_1/\theta_j}{\norm{P_i}_1/\theta_i}B_{uu}\geq p$ and $B_{uv}\leq q$, and (\ref{eq:up1t}) is due to (\ref{eq:p1t}). 
By switching $i$ and $j$,
the foregoing argument also works for the case where $\norm{P_i}_1/\theta_i> \norm{P_j}_1/\theta_j$.
 Hence,
\begin{equation}
\min_{z(i)\neq z(j)}\norm{\bar{P}_i-\bar{P}_j}_1\geq  \frac{p-q}{4\alpha\beta kp}. \label{eq:ngap}
\end{equation}

Let $\hat{z}\in(\{0\}\cup[k])^n$ and $\hat{v}_1,...,\hat{v}_k\in\mathbb{R}^n$ denote a solution to the optimization problem (\ref{eq:opt-km}) (with all nodes in $S_0$ assigned to the $0\Th$ community). 
Define matrix $\hat{V}\in \reals^{n\times n}$ with the $i\Th$ row $\hat{V}_i=\hat{v}_{\hat{z}(i)}$.  
If $\hat{z}(i)=0$, set $\hat{V}_i$ as the zero vector.
Define $S=\{i\in[n]:\norm{\hat{V}_i-\bar{P}_i}_1\geq \frac{p-q}{8\alpha\beta kp}\}$ and recall $S_0=\{i\in[n]:\norm{\hat{P}_i}_1=0\}$. 
Then, by the separation condition (\ref{eq:ngap}) and Lemma \ref{lem:S-anderson}, we have
\begin{equation}
\min_{\pi\in\Pi_k}\sum_{i:\hat{z}(i)\neq \pi(z(i))}\theta_i\leq (2\beta^2+1)\sum_{i\in S}\theta_i+\sum_{i\in S_0}\theta_i.\label{eq:s+s0}
\end{equation} 

In what follows, we derive bounds for $\sum_{i\in S}\theta_i$ and $\sum_{i\in S_0}\theta_i$, respectively. 
Recall that $\tilde{P}_i=\norm{\hat{P}_i}_1^{-1}\hat{P}_i$.
By the definition of $\hat{z}$ and $\hat{V}$, we have
\begin{equation}
\sum_{i=1}^n\norm{\hat{P}_i}_1\norm{\hat{V}_i-\tilde{P}_i}_1\leq (1+\epsilon)\sum_{i=1}^n\norm{\hat{P}_i}_1\norm{\bar{P}_i-\tilde{P}_i}_1.
\label{eq:by-def-km}
\end{equation}

In order to bound $\sum_{i\in S}\theta_i$, we first derive a bound for $\sum_{i\in S}\norm{\hat{P}_i}_1$. That is,
\begin{eqnarray}
\label{eq:y.1}\sum_{i\in S}\norm{\hat{P}_i}_1 &\leq& \frac{8\alpha\beta kp}{p-q}\sum_{i\in S}\norm{\hat{P}_i}_1\norm{\hat{V}_i-\bar{P}_i}_1 \\
\nonumber
& \leq &
\frac{8\alpha\beta kp}{p-q}\sum_{i\in S}
\pth{\norm{\hat{P}_i}_1\norm{\hat{V}_i-\tilde{P}_i}_1
+\norm{\hat{P}_i}_1\norm{\bar{P}_i-\tilde{P}_i}_1}
\\
\label{eq:y.2}&\leq& \frac{8(2+\epsilon)\alpha\beta kp}{p-q}\sum_{i=1}^n\norm{\hat{P}_i}_1\norm{\bar{P}_i-\tilde{P}_i}_1 \\
\label{eq:y.3}&\leq& \frac{16(2+\epsilon)\alpha\beta kp}{p-q}\sum_{i=1}^n\norm{\hat{P}_i-P_i}_1 \\
\label{eq:bs1} &\leq& \frac{16(2+\epsilon)\alpha\beta nkp}{p-q}\fnorm{\hat{P}-P},
\end{eqnarray}
where (\ref{eq:y.1}) uses the definition of $S$, (\ref{eq:y.2}) is by the inequality (\ref{eq:by-def-km}), and (\ref{eq:y.3}) is by the inequality $\norm{\norm{x}_1^{-1}x-\norm{y}_1^{-1}y}_1\leq \frac{2\norm{x-y}_1}{\norm{x}_1\vee\norm{y}_1}$ which in turn is due to the triangle inequality.

Now we are ready to bound $\sum_{i\in S}\theta_i$ as
\begin{eqnarray}
\label{eq:g.1}\sum_{i\in S}\theta_i &\leq& \frac{2\beta k}{pn}\sum_{i\in S}\norm{P_i}_1 \\
\label{eq:g.11}
&\leq& \frac{2\beta k}{pn}\sum_{i\in S}\left(\norm{\hat{P}_i}_1+\norm{\hat{P}_i-P_i}_1\right) \\
\label{eq:g.2}
&\leq& 
\frac{2\beta k}{pn}\left(\frac{16(2+\epsilon)\alpha\beta nkp}{p-q}\fnorm{\hat{P}-P}+n\fnorm{\hat{P}-P}\right) \\
\label{eq:g.3}&\leq& \frac{(66+32\epsilon)\alpha\beta^2k^2}{p-q}\fnorm{\hat{P}-P},
\end{eqnarray}
where (\ref{eq:g.1}) is by the inequality (\ref{eq:p1t}), \eqref{eq:g.11} is due to the triangle inequality, (\ref{eq:g.2}) uses \eqref{eq:bs1} and Cauchy-Schwarz, and (\ref{eq:g.3}) holds since $\alpha, \beta, k \geq 1$.

We now turn to bounding $\sum_{i\in S_0}\theta_i$. To this end, simple algebra leads to
\begin{eqnarray}
\label{eq:s01}\sum_{i\in S_0}\theta_i &\leq& \frac{2\beta k}{pn}\sum_{i\in S_0}\norm{P_i}_1 \\
\label{eq:s02}&\leq& \frac{2\beta k}{pn}\sum_{i=1}^n\norm{\hat{P}_i-P_i}_1 \\
&\leq& \frac{2\beta k}{p}\fnorm{\hat{P}-P} 
\leq
\frac{\alpha\beta^2 k^2}{p-q}\fnorm{\hat{P}-P},
\label{eq:s03}
\end{eqnarray}
where (\ref{eq:s01}) is by the inequality (\ref{eq:p1t}), (\ref{eq:s02}) uses the definition of $S_0$ and \eqref{eq:s03} is due to the Cauchy-Schwarz inequality.


Combining the bounds in (\ref{eq:g.3}), (\ref{eq:s03}) and (\ref{eq:s+s0}), we have
\begin{equation}
\min_{\pi\in\Pi_k}\sum_{\{i:\hat{z}(i)\neq \pi(z(i))\}}\theta_i
\leq 
\frac{C(1+\epsilon)k^2}{p-q}\fnorm{\hat{P}-P},
\label{eq:alast}
\end{equation}
where we have absorbed $\alpha$ and $\beta$ into the constant $C$.
 By Lemma \ref{lem:P-hat-bound}, we have
$$\min_{\pi\in\Pi_k}\sum_{\{i:\hat{z}(i)\neq \pi(z(i))\}}\theta_i\leq C\frac{(1+\epsilon)k^{5/2}\sqrt{n\norm{\theta}_{\infty}^2p+1}}{p-q},$$
with probability at least $1-n^{-(1+C')}$. 
This completes the proof.
\end{proof}

\begin{proof}[Proof of Corollary \ref{cor:ini}]
Under the condition $\min_i\theta_i=\Omega(1)$, the loss $\min_{\pi\in\Pi_k}\sum_{\{i:\hat{z}(i)\neq \pi(z(i))\}}\theta_i$ can be lower bounded by $n\ell(\hat{z},z)$ multiplied by a constant. Moreover, since $p\geq n^{-1}$, $k=O(1)$ and $\norm{\theta}_{\infty}=O(1)$, the rate $\frac{k^{5/2}\sqrt{n\norm{\theta}_{\infty}^2p+1}}{p-q}$ is bounded by $O\left(\frac{\sqrt{np}}{p-q}\right)=O\left(\sqrt{n}|\sqrt{p}-\sqrt{q}|^{-1}\right)$. Thus, it is sufficient to show $n^{-1/2}|\sqrt{p}-\sqrt{q}|^{-1}=O(I^{-1/2})$. This is true by observing that
$$e^{-I}\geq\frac{1}{n}\sum_{i=1}^n\exp\left(-\theta_i\frac{n}{k}(\sqrt{p}-\sqrt{q})^2\right)\geq \exp\left(-\norm{\theta}_{\infty}\frac{n}{k}(\sqrt{p}-\sqrt{q})^2\right).$$
Thus, the proof is complete.
\end{proof}

\subsection{Proofs of Theorem \ref{thm:eq-size-I}, Theorem \ref{thm:algo-J} and Corollary \ref{coro:I-J}}

Now we are going to give proofs of Theorem \ref{thm:eq-size-I}, Theorem \ref{thm:algo-J} and Corollary \ref{coro:I-J}. Note that both Theorem \ref{thm:eq-size-I} and Corollary \ref{coro:I-J} are direct consequences of Theorem \ref{thm:algo-J}. The main argument in the proof of Theorem \ref{thm:algo-J} is the following lemma.

\begin{lemma}
Suppose $1<p/q=O(1)$ and $\delta=o\left(\frac{p-q}{p}\right)$. If there exist two sequences $\gamma_1=o\left(\frac{p-q}{kp}\right)$ and $\gamma_2=o\left(\frac{p-q}{k^2p}\right)$, a constant $C_1>0$ and permutations $\{\pi_i\}_{i\in[n]}\subset \Pi_k$ such that
\begin{equation}
\min_{i\in[n]}\mathbb{P}\left(\frac{1}{n}\sum_{j=1}^n\theta_j\indc{\hat{z}_{-i}^0(j)\neq \pi_i(z(j))}\leq\gamma_1, \frac{1}{n}\sum_{j=1}^n\indc{\hat{z}_{-i}^0(j)\neq \pi_i(z(j))}\leq\gamma_2\right)\geq 1-n^{-(1+C_1)}, \label{eq:ini-zan}
\end{equation}
uniformly for all probability distributions in $\mathcal{P}_n'(\theta,p,q,k,\beta;\delta,\alpha)$.
Then, we have for all $i\in[n]$,
$$\mathbb{P}\left(\hat{z}_{-i}^0(i)\neq \pi_i(z(i))\right)\leq (k-1)\exp\left(-(1-\eta)\theta_i(n_{(1)}+n_{(2)})J_{t^*}(p,q)/2\right)+n^{-(1+C_1)}$$
uniformly for all probability distributions in $\mathcal{P}_n'(\theta,p,q,k,\beta;\delta,\alpha)$, where $\eta=o(1)$.
\end{lemma}
\begin{proof}
In what follows, let $E_i$ denote the event in (\ref{eq:ini-zan}). We are going to derive a bound for $\mathbb{P}\left(\hat{z}_{-1}^0(1)\neq \pi_1(z(1))\text{ and }E_1\right)$. For the sake of brevity, we are going to use $\hat{z}$ and $z$ to denote $\hat{z}_{-1}^0$ and $\pi_1(z)$ in the proof with slight abuse of notaion.
Define $n_u=|\{i\in[n]:z(i)=u\}|$, $m_u=|\{i\in[n]:\hat{z}(i)=u\}|$ and $\hat{\Theta}_u=\sum_{\{i:\hat{z}(i)=u,z(i)=u\}}\theta_i$. Without loss of generality, consider the case $z(i)=1$. Then,
$$\mathbb{P}(\hat{z}(1)\neq 1\text{ and }E_1)\leq \sum_{l=2}^k\mathbb{P}(\hat{z}(1)= l\text{ and }E_1).$$
The arguments for bounding $\mathbb{P}(\hat{z}(1)= l\text{ and }E_1)$ are the same for $l=2,...,k$. Thus, we only give the bound for $l=2$ in details. By the definition, we have
\begin{equation}
\mathbb{P}(\hat{z}(1)= 2\text{ and }E_1)\leq \mathbb{P}\left(\frac{1}{m_2}\sum_{\{i:\hat{z}(i)=2\}}A_{1i}\geq\frac{1}{m_1}\sum_{\{i:\hat{z}(i)=1\}}A_{1i}\text{ and }E_1\right).\label{eq:CR7}
\end{equation}
Define independent random variables $X_i\sim \text{Bernoulli}(\theta_1\theta_iq)$, $Y_i\sim\text{Bernoulli}(\theta_1\theta_ip)$, and $Z_i\sim \text{Bernoulli}(\theta_1\theta_ip)$ for all $i\in[n]$. Then, a stochastic order argument bounds the right hand  side of (\ref{eq:CR7}) by
\begin{equation}
\mathbb{P}\left(\frac{1}{m_2}\sum_{\{i:\hat{z}(i)=2,z(i)=2\}}X_i+\frac{1}{m_2}\sum_{\{i:\hat{z}(i)=2,z(i)=1\}}Z_i\geq\frac{1}{m_1}\sum_{\{i:\hat{z}(i)=1,z(i)=1\}}Y_i\text{ and }E_1\right). \label{eq:M10}
\end{equation}
Using Chernoff bound, for any $\lambda>0$, we upper bound (\ref{eq:M10}) by
\begin{eqnarray}
\nonumber && \mathbb{E}\left\{ \prod_{\{i:\hat{z}(i)=2,z(i)=2\}}(\theta_1\theta_iqe^{\lambda/m_2}+1-\theta_1\theta_iq)\prod_{\{i:\hat{z}(i)=2,z(i)=1\}}(\theta_1\theta_i\alpha pe^{\lambda/m_2}+1-\theta_1\theta_i\alpha p) \right.\\
\nonumber && \left.\prod_{\{i:\hat{z}(i)=1,z(i)=1\}}(\theta_1\theta_ipe^{-\lambda/m_1}+1-\theta_1\theta_ip)\indc{E_1}\right\} \\
\nonumber &\leq& \mathbb{E}\left\{\exp\left(\sum_{\{i:\hat{z}(i)=2,z(i)=2\}}(\theta_1\theta_iqe^{\lambda/m_2}-\theta_1\theta_iq)+\sum_{\{i:\hat{z}(i)=2,z(i)=1\}}(\theta_1\theta_i\alpha pe^{\lambda/m_2}-\theta_1\theta_i\alpha p)\right)\right. \\
\nonumber && \left. \exp\left(\sum_{\{i:\hat{z}(i)=1,z(i)=1\}}(\theta_1\theta_ipe^{-\lambda/m_1}-\theta_1\theta_ip)\right) \indc{E_1}\right\} \\
\label{eq:an1} &=& \mathbb{E}\left\{ \exp\left(\theta_1m_2q(e^{\lambda/m_2}-1)+\theta_1m_1p(e^{-\lambda/m_1}-1)\right)  \indc{E_1}\right\} \\
\label{eq:an2} && \times \mathbb{E}\left\{ \exp\left((\hat{\Theta}_2-m_2)\theta_1q(e^{\lambda/m_2}-1\right)  \indc{E_1}\right\} \\
\label{eq:an3} && \times \mathbb{E}\left\{  \exp\left((\hat{\Theta}_1-m_1)\theta_1p(e^{-\lambda/m_1}-1)\right) \indc{E_1}\right\} \\
\label{eq:an4} && \times \mathbb{E}\left\{  \exp\left(\sum_{\{i:\hat{z}(i)=2,z(i)=1\}}\theta_1\theta_i\alpha p(e^{\lambda/m_2}-1)\right) \indc{E_1}\right\}.
\end{eqnarray}
In what follows, we set
$$\lambda=\frac{m_1m_2}{m_1+m_2}\log\frac{p}{q},$$
We are going to give bounds for the four terms (\ref{eq:an1}), (\ref{eq:an2}), (\ref{eq:an3}) and (\ref{eq:an4}), respectively. On the event $E_1$,
$$|\hat{\Theta}_2-m_2|\leq\left|\sum_{\{i:z(i)=2\}}\theta_i-n_2\right|+ |n_2-m_2|+\sum_{\{i:z(i)=2,\hat{z}(i)=1\}}\theta_i\leq \left(\gamma_1+\gamma_2+\frac{\delta\beta}{k}\right)n,$$
and
$$q|e^{\lambda/m_2}-1|= p^{\frac{m_1}{m_1+m_2}}q^{\frac{m_2}{m_1+m_2}}-q\leq p-q,$$
we have
$$\mathbb{E}\left\{ \exp\left((\hat{\Theta}_2-m_2)\theta_1q(e^{\lambda/m_2}-1\right)  \indc{E_1}\right\}\leq \exp\left(n\left(\gamma_1+\gamma_2+\frac{\delta\beta}{k}\right)\theta_1(p-q)\right),$$
which is a bound for (\ref{eq:an2}). A similar argument leads to a bound (\ref{eq:an3}), which is
$$\mathbb{E}\left\{  \exp\left((\hat{\Theta}_1-m_1)\theta_1p(e^{-\lambda/m_1}-1)\right) \indc{E_1}\right\}\leq\exp\left(n\left(\gamma_1+\gamma_2+\frac{\delta\beta}{k}\right)\theta_1(p-q)\right).$$
The last term (\ref{eq:an4}) has a bound
$$\mathbb{E}\left\{  \exp\left(\sum_{\{i:\hat{z}(i)=2,z(i)=1\}}\theta_1\theta_i\alpha p(e^{\lambda/m_2}-1)\right) \indc{E_1}\right\}\leq \exp\left(n\gamma_1\alpha\theta_1(p-q)\right).$$
Finally, we need a bound for (\ref{eq:an1}). With the current choice of $\lambda$,
$$-m_2q(e^{\lambda/m_2}-1)-m_1p(e^{-\lambda/m_1}-1)=\frac{1}{2}(m_1+m_2)J_{\frac{m_1}{m_1+m_2}}(p,q).$$
Note that
\begin{eqnarray*}
&& \left|(m_1+m_2)J_{\frac{m_1}{m_1+m_2}}(p,q)-(n_1+n_2)J_{\frac{n_1}{n_1+n_2}}(p,q)\right| \\
&\leq& |n_1+n_2-m_1-m_2|J_{\frac{n_1}{n_1+n_2}}(p,q) + (m_1+m_2)\left|J_{\frac{m_1}{m_1+m_2}}(p,q)-J_{\frac{n_1}{n_1+n_2}}(p,q)\right| \\
&\leq& n\gamma_2J_{\tau}(p,q)+n(1+\gamma_2)\left|J_{\tau}(p,q)-J_{\hat{\tau}}(p,q)\right|,
\end{eqnarray*}
where $\tau=\frac{n_1}{n_1+n_2}$ and $\hat{\tau}=\frac{m_1}{m_1+m_2}$.
We will give a bound for $\left|J_{\tau}(p,q)-J_{\hat{\tau}}(p,q)\right|$. Since $\left|\frac{\partial}{\partial \tau}J_{\tau}(p,q)\right|=\frac{1}{2}\left|(p-q)-p^tq^{1-t}\log\frac{p}{q}\right|\leq \frac{1}{2}|p-q|+\frac{1}{2}p|\log p-\log q|\leq |p-q|$, we have $\left|J_{\tau}(p,q)-J_{\hat{\tau}}(p,q)\right|\leq |p-q||\tau-\hat{\tau}|\leq \beta k\gamma_2(p-q)$. Hence, we have a bound for (\ref{eq:an1}), which is
\begin{eqnarray*}
&& \mathbb{E}\left\{ \exp\left(\theta_1m_2q(e^{\lambda/m_2}-1)+\theta_1m_1p(e^{-\lambda/m_1}-1)\right)  \indc{E_1}\right\} \\
&\leq& \exp\left(-\frac{1}{2}\theta_1(n_1+n_2)J_{\frac{n_1}{n_1+n_2}}(p,q)+\frac{1}{2}\theta_1n\gamma_2J_{\frac{n_1}{n_1+n_2}}(p,q)+\frac{1}{2}\theta_1n(1+\gamma_2)k\beta\gamma_2(p-q)\right).
\end{eqnarray*}
Combining the above bounds for (\ref{eq:an1}), (\ref{eq:an2}), (\ref{eq:an3}) and (\ref{eq:an4}), we have
\begin{eqnarray*}
&& \mathbb{P}\left(\hat{z}(1)=2\text{ and }E_1\right)\\
 &\leq& \exp\left(-\frac{1}{2}\theta_1(n_1+n_2)J_{\frac{n_1}{n_1+n_2}}(p,q)\right) \\
&& \times \exp\left(-\frac{1}{2}\theta_1n\gamma_2J_{\frac{n_1}{n_1+n_2}}(p,q)+\left[(2+\alpha)\gamma_1+(2+k\beta)\gamma_2+\frac{2\delta\beta}{k}\right]\theta_1n(p-q)\right).
\end{eqnarray*}
By the property of $J_t(p,q)$ stated in Lemma \ref{lem:J_expand}, $J_{\frac{n_1}{n_1+n_2}}(p,q)\geq (4\beta^2)^{-1}\frac{(p-q)^2}{p}$. Then, under the assumptions $\gamma_1=o\left(\frac{p-q}{p}\right)$, $\gamma_2=o\left(\frac{p-q}{pk}\right)$ and $\delta=o\left(\frac{k(p-q)}{p}\right)$, we have
$$\mathbb{P}\left(\hat{z}(1)=2\text{ and }E_1\right)\leq \exp\left(-\frac{1}{2}(1-\eta)\theta_1(n_1+n_2)J_{\frac{n_1}{n_1+n_2}}(p,q)\right),$$
for some $\eta=o(1)$. The same bound also holds for $\mathbb{P}\left(\hat{z}(1)=l\text{ and }E_1\right)$ for $l=2,...,k$. Thus, a union bound argument gives
$$\mathbb{P}(\hat{z}(1)\neq 1\text{ and }E_1)\leq (k-1)\exp\left(-\frac{1}{2}(1-\eta)\theta_1(n_1+n_2)J_{\frac{n_1}{n_1+n_2}}(p,q)\right).$$
Hence,
$$\mathbb{P}(\hat{z}(1)\neq 1)\leq (k-1)\exp\left(-\frac{1}{2}(1-\eta)\theta_1(n_1+n_2)J_{\frac{n_1}{n_1+n_2}}(p,q)\right)+n^{-(1+C_1)}.$$
Now let us use the original notation and apply the above argument for each node, which leads to the bound
$$\mathbb{P}\left(\hat{z}_{-i}^0(i)\neq \pi_i(z(i))\right)\leq (k-1)\exp\left(-\frac{1}{2}(1-\eta)\theta_i\min_{u\neq v}\left[(n_u+n_v)J_{\frac{n_u}{n_u+n_v}}(p,q)\right]\right)+n^{-C_1},$$
for all $\in[n]$.
By the property of $J_t(p,q)$ stated in Lemma \ref{lem:f_increasing}, $\min_{u\neq v}\left[(n_u+n_v)J_{\frac{n_u}{n_u+n_v}}(p,q)\right]= (n_{(1)}+n_{(2)})J_{t^*}(p,q)$, with $t^*$ specified by (\ref{eq:J-def}). Thus, the proof is complete.
\end{proof}

\begin{proof}[Proof of \prettyref{thm:algo-J}]
It is sufficient to check that the initial clustering step satisfies (\ref{eq:ini-zan}) with $\gamma_1=o\left(\frac{p-q}{kp}\right)$ and $\gamma_2=o\left(\frac{p-q}{k^2p}\right)$. This can be done using the bound in Lemma \ref{lem:initial3} under the assumptions (\ref{eqn:algo-J1}) and (\ref{eqn:algo-J2}). 
Note that the $n$ initial clustering results $\{\hat{z}_{-i}^0\}$ may not correspond to the same permutation. This problem can be taken care of by the consensus step (\ref{eq:consensus}). Details of the argument are referred to the proof of Theorem 2 in \cite{gao2015achieving}.
\end{proof}

\begin{proof}[Proofs of Theorem \ref{thm:eq-size-I} and Corollary \ref{coro:I-J}]
Theorem \ref{thm:eq-size-I} is a direct implication of Theorem \ref{thm:algo-J} by observing $I=J$ when $\beta=1$. The fact that $(n_{(1)}+n_{(2)})J_{t^*}(p,q)\geq 2n_{(1)}J_{1/2}(p,q)\geq \frac{2n}{\beta k}(\sqrt{p}-\sqrt{q})^2$ by Lemma \ref{lem:zongming} implies the result for the case $k\geq 3$ in Corollary \ref{coro:I-J}. For $k=2$, observe that
$$\frac{1}{n}\sum_{i=1}^n\exp\left(-\theta_i\frac{n}{2\beta}(\sqrt{p}-\sqrt{q})^2\right)\leq \left[\frac{1}{n}\sum_{i=1}^n\exp\left(-\theta_i\frac{n}{2}(\sqrt{p}-\sqrt{q})^2\right)\right]^{\beta}.$$
This implies the result for $k=2$ in  Corollary \ref{coro:I-J}.
\end{proof}

\section{Properties of {$J_t(p,q)$}}\label{sec:property}

In this section, we study the quantity $J_t(p,q)$ defined in (\ref{eq:Jtpq}). We will state some lemmas about some useful properties of $J_t(p,q)$ that we have used in the paper. Recall that for $p,q,t\in(0,1)$,
$$J_t(p,q)=2(tp+(1-t)q-p^tq^{1-t}).$$

\begin{lemma}\label{lem:f_increasing}
Given $p,q\in(0,1)$, let $f(x_1,x_2)=x_1p+x_2q-(x_1+x_2)p^\frac{x_1}{x_1+x_2}q^\frac{x_2}{x_1+x_2}$ where $x_1, x_2> 0$. Then the function $f$ is increasing in terms of $x_1$ and $x_2$, respectively.
\end{lemma}
\begin{proof}
By differentiating $f$ against $x_1$ we get
\begin{align*}
\frac{\partial f(x_1,x_2)}{\partial x_1}&=p-q\left(\frac{p}{q}\right)^\frac{x_1}{x_1+x_2}-q\left(\frac{p}{q}\right)^\frac{x_1}{x_1+x_2}\log \left(\frac{p}{q}\right)\frac{x_2}{x_1+x_2}.
\end{align*}
Thus $\lim_{x_1\rightarrow\infty}\frac{\partial f(x_1,x_2)}{\partial x_1}=0$.
Moreover,
\begin{align*}
\frac{\partial^2 f(x_1,x_2)}{\partial x_1^2}=-q\left(\frac{p}{q}\right)^\frac{x_1}{x_1+x_2}\log^2\left(\frac{p}{q}\right)\frac{x_2^2}{(x_1+x_2)^3}\leq 0,
\end{align*}
Therefore, $\frac{\partial f(x_1,x_2)}{\partial x_1}\geq 0$ for all $x_1,x_2> 0$. This shows $f(x_1,x_2)$ is increasing with respect to $x_1$. Similarly we can prove that $f(x_1,x_2)$ is also an increasing function in terms of $x_2$.
\end{proof}

\begin{lemma}\label{lem:t_asymmetric}
For any $0<q<p<1$ and $0<t\leq\frac{1}{2}$, we have
$$
J_t(p,q)\leq J_{1-t}(p,q).
$$
\end{lemma}
\begin{proof}
Define $S(t)=\frac{1}{2}\left(J_t(p,q)-J_{1-t}(p,q)\right)=(2t-1)(p-q)-\left(q(\frac{p}{q})^t-p(\frac{q}{qp})^t\right)$. Then, we have
$$
S''(t)=-\log^2\left(\frac{q}{p}\right)\left(p^{t}q^{1-t}-p^{1-t}q^{t}\right)\geq 0.
$$
Since $S(0)=S(1/2)=0$, we have $S(t)\leq 0$ for all $t\in (0,1/2]$.
\end{proof}

\begin{lemma}\label{lem:zongming}
For any $0<q<p<1$ and $0<x_1\leq x_2$, we have
$$2x_1J_{1/2}(p,q)\leq (x_1+x_2)J_{\frac{x_1}{x_1+x_2}}(p,q)\leq (x_1+x_2)J_{1/2}(p,q).$$
\end{lemma}
\begin{proof}
The first inequality $2x_1J_{1/2}(p,q)\leq (x_1+x_2)J_{\frac{x_1}{x_1+x_2}}(p,q)$ is a consequence of Lemma \ref{lem:f_increasing} and $x_1\leq x_2$. Since $\left(\frac{\partial}{\partial t}\right)^2J_t(p,q)=-2p^tq^{1-t}\log^2\left(\frac{p}{q}\right)\leq 0$, $J_t(p,q)$ is concave in $t$. Thus,
\begin{equation}
\frac{1}{2}\left(J_t(p,q)+J_{1-t}(p,q)\right)\leq J_{1/2}(p,q).\label{eq:conc-J}
\end{equation}
When $t\in (0,1/2]$, $J_t(p,q)\leq \frac{1}{2}\left(J_t(p,q)+J_{1-t}(p,q)\right)$ according to Lemma \ref{lem:t_asymmetric}. Thus, $J_t(p,q)\leq J_{1/2}(p,q)$, which leads to the second inequality $(x_1+x_2)J_{\frac{x_1}{x_1+x_2}}(p,q)\leq (x_1+x_2)J_{1/2}(p,q)$ by the assumption $x_1\leq x_2$.
\end{proof}

\begin{lemma}\label{lem:J_expand}
For any $0<p,q,t<1$, we have
\begin{equation}
2\min(t,1-t)(\sqrt{p}-\sqrt{q})^2\leq J_t(p,q)\leq 2(\sqrt{p}-\sqrt{q})^2.\label{eq:tiamat}
\end{equation}
Moreover, if $\max(p/q,q/p)\leq M$, then we have
\begin{equation}
J_t(p,q)\leq \left(2+\frac{4M^4}{3}\right)\min(t,1-t)\frac{(p-q)^2}{\min(p,q)}.\label{eq:therion}
\end{equation}
\end{lemma}
\begin{proof}
Without loss of generality, let $p>q$.
We first consider the case $0<t\leq 1/2$. By Lemma \ref{lem:zongming}, we have
$$2\frac{x_1}{x_1+x_2}J_{1/2}(p,q)\leq J_{\frac{x_1}{x_1+x_2}}(p,q)\leq J_{1/2}(p,q).$$
Let $\frac{x_1}{x_1+x_2}=t$, and we have
\begin{equation}
2t(\sqrt{p}-\sqrt{q})^2\leq J_t(p,q)\leq (\sqrt{p}-\sqrt{q})^2.\label{eq:jiegun}
\end{equation}
Now we consider the case $1/2\leq t<1$. Let $s=1-t$. By Lemma \ref{lem:t_asymmetric} and (\ref{eq:jiegun}), we have
$$J_t(p,q)\geq J_s(p,q)\geq 2s(\sqrt{p}-\sqrt{q})^2=2(1-t)(\sqrt{p}-\sqrt{q})^2.$$
Using (\ref{eq:conc-J}) and (\ref{eq:jiegun}), we have
$$J_t(p,q)\leq 2J_{1/2}(p,q)-J_s(p,q)\leq 2J_{1/2}(p,q)-2s(\sqrt{p}-\sqrt{q})^2=2t(\sqrt{p}-\sqrt{q})^2\leq 2(\sqrt{p}-\sqrt{q})^2.$$
Hence,
\begin{equation}
2(1-t)(\sqrt{p}-\sqrt{q})^2\leq J_t(p,q)\leq 2(\sqrt{p}-\sqrt{q})^2.\label{eq:lolu}
\end{equation}
Combine (\ref{eq:jiegun}) and (\ref{eq:lolu}), and we can derive (\ref{eq:tiamat}) for $p>q$. A symmetric argument leads to the same result for $p<q$. When, $p=q$, the result trivially holds. Thus, the proof for (\ref{eq:tiamat}) is complete.

To prove (\ref{eq:therion}), we use the identity
$$\frac{1}{2t(1-t)}J_t(p,q)=p\frac{1}{1-t}\left(1-\left(\frac{q}{p}\right)^{1-t}\right)+q\frac{1}{t}\left(1-\left(\frac{p}{q}\right)^t\right).$$
By Taylor's theorem, we have
$$\frac{1}{\alpha}(1-x^{\alpha})=1-x+\frac{1}{2}(1-\alpha)(x-1)^2-\frac{1}{6}(\alpha-1)(\alpha-2)\xi^{\alpha-3}(x-1)^3,$$
for some $\xi$ between $x$ and $1$. Thus, using the condition that $\max(p/q,q/p)\leq M$, we have
$$\frac{1}{2t(1-t)}J_t(p,q)\leq \left(1+\frac{2M^4}{3}\right)\frac{(p-q)^2}{\min(p,q)}.$$
Then, we can derive (\ref{eq:therion}) by the fact that $t(1-t)\leq \min(t,1-t)$.
\end{proof}

\section{Proofs of Auxiliary Results}\label{sec:pf-aux}

\begin{proof}[Proof of Lemma \ref{lem:t-lower}]
Note that \eqref{eq:hypo} is a simple vs.~simple hypothesis testing problem.
By the Neyman--Pearson lemma, the optimal test is the likelihood ratio test $\phi$ which rejects $H_0$ if
\begin{eqnarray*}
&& \prod_{i=1}^{m}(\theta_0\theta_ip)^{X_i}(1-\theta_0\theta_ip)^{1-X_i}
\prod_{i=m+1}^{2m}(\theta_0\theta_iq)^{X_i}(1-\theta_0\theta_iq)^{1-X_i}
\\
&& ~~~~~~~~~~< 
\prod_{i=1}^{m}(\theta_0\theta_iq)^{X_i}(1-\theta_0\theta_iq)^{1-X_i}
\prod_{i=m+1}^{2m}(\theta_0\theta_ip)^{X_i}(1-\theta_0\theta_ip)^{1-X_i}.
\end{eqnarray*}
Therefore,
\begin{align*}
\p_{H_0}\phi
&=\p\Bigg(\sum_{i=1}^{m}\Big(X_i\log\frac{q(1-\theta_0\theta_ip)}{p(1-\theta_0\theta_iq)}-\log\frac{1-\theta_0\theta_ip}{1-\theta_0\theta_iq}\Big)
+\sum_{i=m+1}^{2m}\Big(X_i\log\frac{p(1-\theta_0\theta_iq)}{q(1-\theta_0\theta_ip)}-\log\frac{1-\theta_0\theta_iq}{1-\theta_0\theta_ip}\Big)>0\Bigg).
\end{align*}
To establish the desired bound for this quantity, we employ below a refined version of the Cramer--Chernoff argument \cite[Proposition 14.23]{van2000}. 
To this end, 
for any fixed $t>0$, define independent random variables $\{W_i\}_{i=1}^{2m}$ by
\begin{equation*}
\Prob\left(W_i=t\log\frac{q}{p}\right)=\theta_0\theta_ip,\quad \Prob\left(W_i=t\log\frac{1-\theta_0\theta_iq}{1-\theta_0\theta_ip}\right)=1-\theta_0\theta_ip,\quad \mbox{for $i=1,...,m$}, 
\end{equation*}
and
\begin{equation*}
\Prob\left(W_i=t\log\frac{p}{q}\right)=\theta_0\theta_iq,\quad \Prob\left(W_i=t\log\frac{1-\theta_0\theta_ip}{1-\theta_0\theta_iq}\right)=1-\theta_0\theta_iq,\quad \mbox{for $i=m+1,...,2m$}.
\end{equation*}
In addition, let
\begin{equation*}
B_i = \begin{cases}
(\theta_0\theta_ip)^{1-t}(\theta_0\theta_iq)^t+(1-\theta_0\theta_ip)^{1-t}(1-\theta_0\theta_iq)^t, & \quad i=1,\dots,m;\\
(\theta_0\theta_iq)^{1-t}(\theta_0\theta_ip)^t+(1-\theta_0\theta_iq)^{1-t}(1-\theta_0\theta_ip)^t, & \quad i=m+1,\dots, 2m.
\end{cases}
\end{equation*}
We lower bound $P_{H_0}\phi$ by
\begin{eqnarray*}
P_{H_0}\phi &=& \mathbb{P}\left(\sum_{i=1}^{m}W_i+\sum_{i=m+1}^{2m}W_i>0\right) \\
&\geq& \sum_{0<\sum_i w_i<L}\prod_{i=1}^{2m}\Prob(W_i=w_i) \\
&\geq& \pth{\prod_{i=1}^{2m}B_i} e^{-L}  \sum_{0<\sum_i w_i<L}\prod_{i=1}^{2m}\frac{P_i(w_i)e^{w_i}}{B_i} \\
&=&  \pth{\prod_{i=1}^{2m}B_i }e^{-L}\sum_{0<\sum_i w_i<L}\prod_{i=1}^{2m}Q_i(w_i) \\
&=& \pth{\prod_{i=1}^{2m} B_i} e^{-L}\mathbb{Q}\left(0<\sum_{i=1}^{2m} W_i<L\right),
\end{eqnarray*}
where
$$Q_i\left(W_i=t\log\frac{q}{p}\right)=\frac{(\theta_0\theta_ip)^{1-t}(\theta_0\theta_iq)^t}{B_i},\quad Q_i\left(W_i=t\log\frac{1-\theta_0\theta_iq}{1-\theta_0\theta_ip}\right)=\frac{(1-\theta_0\theta_ip)^{1-t}(1-\theta_0\theta_iq)^t}{B_i}$$
for $i=1,...,m$ and
$$Q_i\left(W_i=t\log\frac{p}{q}\right)=\frac{(\theta_0\theta_iq)^{1-t}(\theta_0\theta_ip)^t}{B_i},\quad Q_i\left(W_i=t\log\frac{1-\theta_0\theta_ip}{1-\theta_0\theta_iq}\right)=\frac{(1-\theta_0\theta_iq)^{1-t}(1-\theta_0\theta_ip)^t}{B_i}$$
for $i=m+1,...,2m$. 
We have also used the abbreviations $P_i(w_i) = \Prob(W_i=w_i)$ and $Q_i(w_i) = \bbQ(W_i=w_i)$.

To obtain the desired lower bound, we set $t$ to be the minimizer of $\prod_{i=1}^{2m}B_i$.
Since the minimizer is a stationary point, it satisfies
\begin{equation}
\sum_{i=1}^{2m}\mathbb{E}_{\bbQ}W_i=0.\label{eq:t-mean}
\end{equation}
For any $t,a,b\in(0,1)$, recall the definition of $J_t(a,b)$ in \eqref{eq:Jtpq}.
By \prettyref{lem:J_expand}, we have
\begin{equation}
J_t(1-a,1-b)\leq CaJ_t(a,b),\label{eq:lvlv}
\end{equation}
where $C$ only depends on the ratio $a/b$.
Therefore, under the condition $a\asymp b=o(1)$, \eqref{eq:lvlv} implies
\begin{equation}
\log\left(1-\frac{1}{2} J_t(a,b)-\frac{1}{2}J_t(1-a,1-b)\right)\geq -\frac{1}{2}(1+\eta)J_t(a,b),\label{eq:clear}
\end{equation}
for some $\eta=o(1)$ independent of $t$.
Using (\ref{eq:clear}), under the assumption that $1<p/q=O(1)$, we have
\begin{eqnarray}
\nonumber\prod_{i=1}^{2m} B_i &\geq& \exp\left(-\frac{1+\eta}{2}\sum_{i=1}^{m}J_t(\theta_0\theta_iq,\theta_0\theta_ip)-\frac{1+\eta}{2}\sum_{i=m+1}^{2m}J_t(\theta_0\theta_ip,\theta_0\theta_iq)\right) \\
\label{eq:jingran} &=& \exp\left(-(1+\eta){\theta_0 m}
\left(p+q-p^{1-t}q^t-q^{1-t}p^t\right)
\right).
\end{eqnarray}
Hence,
\begin{eqnarray*}
\min_{0\leq t\leq 1}\prod_{i=1}^{2m}B_i &\geq& \exp\left(-(1+\eta){\theta_0 m}\max_{0\leq t\leq 1}\left(p+q-p^{1-t}q^t-q^{1-t}p^t\right)\right) \\
&=& \exp\left(-(1+\eta){\theta_0 m}(\sqrt{p}-\sqrt{q})^2\right).
\end{eqnarray*}


We now turn to lower bounding $e^{-L}\mathbb{Q}\left(0<\sum_{i=1}^{2m} W_i<L\right)$ with $t$ satisfying (\ref{eq:t-mean}). 
To this end, we first calculate the variances of the $W_i$'s.
For $i=1,...,m$, there exists some constant $C>0$ such that
\begin{eqnarray*}
\text{Var}_{\bbQ}(W_i) &\leq& \mathbb{E}_{\bbQ}(W_i^2) \\
&\leq& \left(t\log\frac{p}{q}\right)^2Q_i\left(t\log\frac{q}{p}\right) + \left(t\log\frac{1-\theta_0\theta_iq}{1-\theta_0\theta_ip}\right)^2Q_i\left(t\log\frac{1-\theta_0\theta_iq}{1-\theta_0\theta_ip}\right) \\
&\leq& C\theta_0\theta_ip\left(\log(\theta_0\theta_ip)-\log(\theta_0\theta_iq)\right)^2+\left(\log(1-\theta_0\theta_ip)-\log(1-\theta_0\theta_iq)\right)^2 \\
&\leq& C\theta_0\theta_ip \frac{(\theta_0\theta_ip-\theta_0\theta_iq)^2}{(\theta_0\theta_iq)^2}+C(\theta_0\theta_ip-\theta_0\theta_iq)^2 \\
&\leq& C\frac{\theta_0\theta_i(p-q)^2}{p}.
\end{eqnarray*}
In addition, we have
\begin{equation*}
\mathbb{E}_{\bbQ}(W_i^2)\geq  \left(t\log\frac{p}{q}\right)^2Q_i\left(t\log\frac{q}{p}\right)\gtrsim \frac{\theta_0\theta_i(p-q)^2}{p},	\quad
\mbox{and} \quad
(\mathbb{E}W_i)^2=o\left(\frac{\theta_0\theta_i(p-q)^2}{p}\right).
\end{equation*}
Similar bounds hold for $W_i$, $i=m+1,...,2m$.
Thus, we obtain that
\begin{equation*}
\sum_{i=1}^{2m}\text{Var}_{\bbQ}(W_i)\asymp \frac{\theta_0 m (p-q)^2}{p}.	
\end{equation*}
Note that with $t\in [\gamma,1-\gamma]$ and $p/q=O(1)$, the value of $W_i$ is bounded by constant, for any $i\in[2m]$. 
Under the assumption that $\theta_0 m(\sqrt{p}-\sqrt{q})^2\rightarrow\infty$, we have $\sum_{i=1}^{2m}\text{Var}_{\bbQ}(W_i)\rightarrow\infty$, implying the indicator function $\indc{|W_i-\mathbb{E}W_i|>\epsilon\sqrt{\sum_{i=1}^{2m}\text{Var}_{\bbQ}(W_i)} }$ goes to 0 for every $i$.

Thus
\begin{equation*}
\lim_{m\rightarrow\infty}
\sum_{i=1}^{2m}\mathbb{E}(W_i-\mathbb{E}W_i)^2\indc{|W_i-\mathbb{E}W_i|>\epsilon\sqrt{\sum_{i=1}^{2m}\text{Var}_{\bbQ}(W_i)} }=0,	
\end{equation*}
for any constant $\epsilon>0$.
Together with \eqref{eq:t-mean}, the Lindeberg condition implies that under $\bbQ$, $\frac{\sum_{i=1}^{2m}W_i}{\sqrt{\sum_{i=1}^{2m}\text{Var}_{\bbQ}(W_i)}}$ converges to $N(0,1)$.
Taking $L=\sqrt{\sum_{i=1}^{2m}\text{Var}_{\bbQ}(W_i)}$, we have that for any $\eta = o(1)$, 
\begin{equation*}
e^{-L}\mathbb{Q}\left(0<\sum_{i=1}^{2m}W_i<L\right)\geq \exp\left(-\eta \theta_0m(\sqrt{p}-\sqrt{q})^2\right)	
\end{equation*}
for sufficiently large values of $m$.
This completes the proof when $\theta_0m(\sqrt{p}-\sqrt{q})^2\rightarrow\infty$.

When $\theta_0 m (\sqrt{p}-\sqrt{q})^2=O(1)$, then we have
\begin{eqnarray*}
\inf_{\phi}(P_{H_0}\phi+P_{H_1}(1-\phi)) &=& \int dP_{H_0}\wedge dP_{H_1} \\
&\geq& \frac{1}{2}\left(\int\sqrt{dP_{H_0}dP_{H_1}}\right)^2 \\
&=& \frac{1}{2}\left(\prod_{i=1}^{2m}\left(\theta_0\theta_i\sqrt{pq}+\sqrt{(1-\theta_0\theta_ip)(1-\theta_0\theta_iq)}\right)\right)^2 \\
&\geq& \frac{1}{2}\exp\left(-(2+\eta)\theta_0 m(\sqrt{p}-\sqrt{q})^2\right) \\
&\geq& c.
\end{eqnarray*}
This completes the proof.
\end{proof}

\begin{proof}[Proof of Lemma \ref{lem:t-upper}]
We bound $P_{H_0}\phi$ by
\begin{eqnarray*}
P_{H_0}\phi &\leq& \left(\prod_{i=n/2+1}^n\mathbb{E}e^{tX_i}\right)\left(\prod_{i=1}^{n/2}\mathbb{E}e^{-tX_i}\right) \\
&=& \exp\left(\sum_{i=n/2+1}^n\log\left(1-\theta_0\theta_iq+\theta_0\theta_iqe^t\right)+\sum_{i=1}^{n/2}\log\left(1-\theta_0\theta_ip+\theta_0\theta_ipe^{-t}\right)\right) \\
&\leq& \exp\left(\sum_{i=n/2+1}^n\left(-\theta_0\theta_iq+\theta_0\theta_iqe^t\right)+\sum_{i=1}^{n/2}\left(-\theta_0\theta_ip+\theta_0\theta_ipe^{-t}\right)\right) \\
&=& \exp\left(-\frac{\theta_0n}{2}\left(p+q-pe^{-t}-qe^t\right)\right) \\
&=& \exp\left(-\frac{\theta_0n}{2}(\sqrt{p}-\sqrt{q})^2\right),
\end{eqnarray*}
where we have set $e^t=\sqrt{p/q}$. The same bound can be established for $P_{H_1}(1-\phi)$.
\end{proof}

\begin{proof}[Proof of Lemma \ref{lem:t-lower-aprox}]
The proof is very similar to that of Lemma \ref{lem:t-lower}. Therefore, we only sketch the difference. Without loss of generality, let $\theta_1\geq \theta_2\geq...\geq\theta_m$, $\theta_{m+1}\geq\theta_{m+2}\geq...\geq\theta_{m+m_1}$, and $m\leq m_1$. Then, we have
$$\inf_{\phi}\left(P_{H_0}\phi+P_{H_1}(1-\phi)\right)\geq \inf_{\phi}\left(P_{\bar{H}_0}\phi+P_{\bar{H}_1}(1-\phi)\right),$$
where $\bar{H}_0$ and $\bar{H}_1$ correspond to the following two hypotheses.
\begin{equation*}
\begin{aligned}
& \bar{H}_0: X\sim \bigotimes_{i=1}^{m}\text{Bern}\left(\theta_0\theta_ip\right) \otimes \bigotimes_{i=m+1}^{2m}\text{Bern}\left(\theta_0\theta_iq\right) \\
& \quad \quad \quad \mbox{vs.}\quad
\bar{H}_1: X\sim \bigotimes_{i=1}^{m}\text{Bern}\left(\theta_0\theta_i q\right) \otimes \bigotimes_{i=m+1}^{2m}\text{Bern}\left(\theta_0\theta_i p\right).	
\end{aligned}
\end{equation*}
Bounding $\inf_{\phi}\left(P_{\bar{H}_0}\phi+P_{\bar{H}_1}(1-\phi)\right)$ is handled by the proof of Lemma \ref{lem:t-lower} except that we do not have the relation (\ref{eq:normalize-test}) exactly. This slightly change the derivation of (\ref{eq:jingran}), as we will illustrate below. By the definition of $J_t(\cdot,\cdot)$, we have
\begin{eqnarray*}
&& \frac{1}{2}\left(\sum_{i=1}^nJ_t(\theta_0\theta_iq,\theta_0\theta_ip)+\sum_{i=m+1}^{2m}J_t(\theta_0\theta_ip,\theta_0\theta_iq)\right) \\
&=& \left(\theta_0\sum_{i=1}^m\theta_i\right)(tq+(1-t)p-q^tp^{1-t}) + \left(\theta_0\sum_{i=m+1}^{2m}\theta_i\right)(tp+(1-t)p-p^tq^{1-t}) \\
&=& \theta_0m(p+q-p^{1-t}q^t-q^{1-t}p^t)+\theta_0m\left|\frac{1}{m}\sum_{i=1}^m\theta_i-1\right|(tq+(1-t)p-q^tp^{1-t}) \\
&& + \theta_0m\left|\frac{1}{m}\sum_{i=m+1}^{2m}\theta_i-1\right|(tp+(1-t)p-p^tq^{1-t}) \\
&\leq& \theta_0m(p+q-p^{1-t}q^t-q^{1-t}p^t)+C\eta\theta_0m(\sqrt{p}-\sqrt{q})^2,
\end{eqnarray*}
for some $\eta=o(1)$. The last inequality uses Lemma \ref{lem:J_expand} and the fact that $\delta=o(1)$. Since the term $C\eta\theta_0m(\sqrt{p}-\sqrt{q})^2$ is of smaller order compared with the targeted exponent, the desired result can be derived following the remaining proof of Lemma \ref{lem:t-lower}.
\end{proof}

\begin{proof}[Proof of Lemma \ref{lem:S-anderson}]
For each $u\in[k]$, we define
$$\mathcal{C}_u=\left\{i\in z^{-1}(u)\cap S_0^c: \norm{\tilde{V}_i-V_i}<b\right\}.$$
Following \cite{chen2015convexified}, we divide the sets $\{\mathcal{C}_u\}_{u\in[k]}$ into three groups. Define
\begin{eqnarray*}
R_1 &=& \{u\in[k]:\mathcal{C}_u=\varnothing\}, \\
R_2 &=& \{u\in[k]:\mathcal{C}_u\neq\varnothing, \forall i,j\in \mathcal{C}_u, \tilde{z}(i)=\tilde{z}(j)\}, \\
R_3 &=& \{u\in[k]:\mathcal{C}_u\neq \varnothing,\exists i,j\in\mathcal{C}_u, \text{s.t. }i\neq j, \tilde{z}(i)\neq\tilde{z}(j)\}.
\end{eqnarray*}
Then, it is easy to see that $\cup_{u\in[k]}\mathcal{C}_u=S_0^c\backslash S^c$ and $\mathcal{C}_u\cap\mathcal{C}_v=\varnothing$ for any $u\neq v$. Suppose there exists some $i\in\mathcal{C}_u$ and $j\in\mathcal{C}_v$ such that $u\neq v$ but $\tilde{z}(i)=\tilde{z}(j)$. Then, by the fact $\tilde{V}_i=\tilde{V}_j$, we have
$$\norm{V_i-V_j}\leq \norm{V_i-\tilde{V}_i}+\norm{V_j-\tilde{V}_j}<2b,$$
contradicting (\ref{eq:tttgap}). This means $\tilde{z}(i)$ and $\tilde{z}(j)$
 take different values if $i$ and $j$ are not in the same $\mathcal{C}_u$'s. By the definition of $R_2$, the nodes in $\cup_{u\in R_2}\mathcal{C}_u$ have the same partition induced by $z$ and $\tilde{z}$. Therefore,
 $$\min_{\pi\in\Pi_k}\sum_{\{i:\hat{z}(i)\neq \pi(z(i))\}}\theta_i\leq\sum_{i\in S_0}\theta_i+\sum_{i\in S}\theta_i+\sum_{i\in \cup_{u\in R_3}\mathcal{C}_u}\theta_i.$$
 It is sufficient to bound $\sum_{i\in \cup_{u\in R_3}\mathcal{C}_u}\theta_i$. By the definition of $R_3$, we observe that each $\mathcal{C}_u$ for some $u\in R_3$ contains at least two different labels given by $\tilde{z}$. Thus
 we have $|R_2|+2|R_3|\leq k$. Moreover, since $k=|R_1|+|R_2|+|R_3|$, we have $|R_3|\leq |R_1|$. This leads to
 \begin{eqnarray*}
 \sum_{i\in \cup_{u\in R_3}\mathcal{C}_u}\theta_i &\leq& |R_3|(1+\delta)\frac{\beta n}{k} \\
 &\leq& |R_1|(1+\delta)\frac{\beta n}{k} \\
 &\leq& \frac{1+\delta}{1-\delta}\beta^2\sum_{i\in\cup_{u\in R_1}\{i\in[n]:z(i)=u\}}\theta_i \\
 &\leq& \frac{1+\delta}{1-\delta}\beta^2\sum_{i\in S}\theta_i\\
 &\leq& 2\beta^2\sum_{i\in S}\theta_i.
 \end{eqnarray*}
 This completes the proof.
\end{proof}

\begin{proof}[Proof of Lemma \ref{lem:P-hat-bound}]
Define the matrix $P'\in\mathbb{R}^{n\times n}$ by $P'_{ij}=\theta_i\theta_jB_{z(i)z(j)}$ for each $i,j\in[n]$. Then, $P'$ has rank at most $k$ and differs from $P$ only by the diagonal entries.
By the definition of $\hat{P}$, we have $\fnorm{\hat{P}-T_{\tau}(A)}^2\leq \fnorm{P'-T_{\tau}(A)}^2$. After rearrangement, we have
\begin{eqnarray*}
\fnorm{\hat{P}-P}^2 &\leq& 2\left|\iprod{\hat{P}-P'}{T_{\tau}(A)-P}\right|+\fnorm{P'-P}^2 \\
&\leq& 2\fnorm{\hat{P}-P'}\sup_{\{K:\fnorm{K}=1,\text{rank}(K)\leq 2k\}}\left|\iprod{K}{T_{\tau}(A)-P}\right| + \fnorm{P'-P}^2 \\
&\leq& \frac{1}{4}\fnorm{\hat{P}-P'}^2+4\sup_{\{K:\fnorm{K}=1,\text{rank}(K)\leq 2k\}}\left|\iprod{K}{T_{\tau}(A)-P}\right|^2+ \fnorm{P'-P}^2 \\
&\leq& \frac{1}{2}\fnorm{\hat{P}-P}^2+\frac{3}{2}\fnorm{P'-P}^2+4\sup_{\{K:\fnorm{K}=1,\text{rank}(K)\leq 2k\}}\left|\iprod{K}{T_{\tau}(A)-P}\right|^2.
\end{eqnarray*}
Therefore,
\begin{equation}
\fnorm{\hat{P}-P}^2\leq 3\fnorm{P'-P}^2+8\sup_{\{K:\fnorm{K}=1,\text{rank}(K)\leq 2k\}}\left|\iprod{K}{T_{\tau}(A)-P}\right|^2.\label{eq:vardy}
\end{equation}
Apply singular value decomposition to $K$ and we get $K=\sum_{l=1}^{2k}\lambda_lu_lu_l^T$. Then,
$$\left|\iprod{K}{T_{\tau}(A)-P}\right|\leq \sum_{l=1}^{2k}|\lambda_l||u_l^T(T_{\tau}(A)-P)u_l|\leq \opnorm{T_{\tau}(A)-P}\sum_{l=1}^{2k}|\lambda_l|\leq \sqrt{2k}\opnorm{T_{\tau}(A)-P}.$$
By Lemma 5 of \cite{gao2015achieving}, $\opnorm{T_{\tau}(A)-P}\leq C\sqrt{n\alpha p\norm{\theta}_{\infty}^2+1}$ with probability at least $1-n^{-C'}$, where the constant $C'$ can be made arbitrarily large. Hence,
$$8\sup_{\{K:\fnorm{K}=1,\text{rank}(K)\leq 2k\}}\left|\iprod{K}{T_{\tau}(A)-P}\right|^2\leq C_1k(n\alpha p\norm{\theta}_{\infty}^2+1),$$
with probability at least $1-n^{-C'}$
Moreover,
$$3\fnorm{P'-P}^2=3\sum_{i=1}^n\theta_i^2B_{z(i)z(i)}^2\leq 3\alpha^2p^2\norm{\theta}_{\infty}n(1+\delta)\leq C_2\alpha^2p\norm{\theta}_{\infty}^2n.$$
Using (\ref{eq:vardy}), the proof is complete by absorbing $\alpha$ into the constant.
\end{proof}

%
%

\bibliographystyle{abbrvnat}
\bibliography{dcbm}

\end{document}